\documentclass[10pt]{article}
\usepackage{amsthm,amsmath,amssymb,amscd}
\usepackage{amssymb,bbm}
\usepackage{graphics}
\usepackage{color}

\newtheorem{Theorem}{Theorem}[section]      
\newtheorem{Proposition}[Theorem]{Proposition}    
\newtheorem{Lemma}[Theorem]{Lemma}            
\newtheorem{Definition}[Theorem]{Definition}
\newtheorem{Corollary}[Theorem]{Corollary}    
\newtheorem{Remark}[Theorem]{Remark}           

\newcommand{\Om}{\Omega}  
\newcommand{\eps}{\varepsilon} 
\newcommand{\R}{\mathbb{R}} 
 
\renewcommand{\div}{\textrm{div}}

\def\qed{\hfill $\square$ \goodbreak \smallskip}

\title
{New examples of extremal domains for the first eigenvalue of the Laplace-Beltrami operator in a Riemannian manifold with boundary}
\author
{Jimmy Lamboley\footnote{Universit\'e Paris Dauphine - CEREMADE, UMR CNRS 7534, Place du Mar\'echal de Lattre de Tassigny, 75775 Paris Cedex 16, France.
{\it e-mail:} lamboley@ceremade.dauphine.fr}, Pieralberto Sicbaldi\footnote{Aix Marseille Universit\'e - CNRS - Ecole Centrale Marseille, I2M, UMR 7373, 13453 Marseille, France.\newline
{\it e-mail:} pieralberto.sicbaldi@univ-amu.fr}}
\date{\today}

\begin{document}

\maketitle

{\bf Abstract.} {We build new examples of} extremal domains with small prescribed volume for the first eigenvalue of the Laplace-Beltrami operator in some Riemannian manifold with boundary. These domains are close to half balls of small radius centered at a nondegenerate critical point of the mean curvature function of the boundary of the manifold, and their boundary intersects the boundary of the manifold orthogonally.

\section{Introduction}

New examples of domains with small prescribed volume that are critical points for the first eigenvalue of the Dirichlet Laplace-Beltrami operator are built in \cite{pacsic}, under the hypothesis that the Riemannian manifold has at least one nondegenerate critical point of the scalar curvature function. In that case, such domains are given by small perturbations of geodesic balls of small radius centered at a nondegenerate critical point of the scalar curvature. This result has been generalized in \cite{del-sic} to all compact Riemannian manifolds by eliminating the hypothesis of the existence of a nondegenerate critical point of the scalar curvature.

\medskip 

Such examples of critical points for the Laplace-Beltrami operator are parallels to similar shape examples of critical points for the area functional, under the same assumptions, which lead to the construction of constant mean curvature small topological spheres, see \cite{Pac-Xu,Ye}.

\medskip

The aim of this paper is to give some new examples of domains $\Omega$ that are critical points for the first eigenvalue of the Laplace-Beltrami operator (i.e. extremal domains) in some Riemannian manifolds $M$ with boundary. Such examples are new because the boundary of the domain is partially included in the boundary of the manifold. The domains we obtain are close to half-balls centered at a point of $\partial M$ where the mean curvature of $\partial M$ is critical and the criticality is not degenerate. In particular, in the simplest situation, $M$ can be a domain of the Euclidean space, see Fig. \ref{fig1}. Again, we can make a parallel with the case of the area,  for which a similar result has been proven in the Euclidian case and dimension 3 in \cite{fall}, though it is expected to be valid in the general case.

\begin{figure}[!ht]
\centering
{\scalebox{.5}{\input{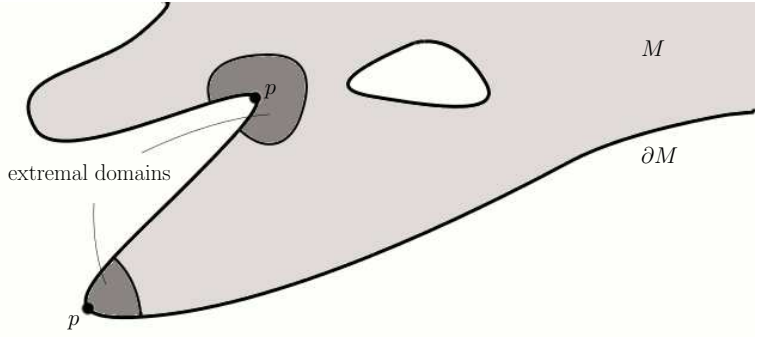}}}
\caption{$M$ can be a Euclidean domain (bounded or not). If $p$ is a nondegenerate critical point for the mean curvature of $\partial M$, then it is possible to construct an extremal domain as a perturbation of a half-ball centered at $p$.}
\label{fig1}
\end{figure}

\medskip

Assume that we are given $(M,g)$ an $(n+1)$-dimensional Riemannian manifold, $n \geq 1$, with boundary $\partial M \neq \emptyset$. The boundary $\partial M$ is a smooth $n$-dimensional Riemannian manifold with the metric $\tilde{g}$ induced by $g$. For a domain $\Omega$ contained in the interior of $M$, $\Omega \subset \mathring M$, the first eigenvalue of the Laplace-Beltrami operator with 0 Dirichlet boundary condition is then given by
\[
\lambda_{\Omega} = \displaystyle \min_{u \in H_0^1(\Omega)} \frac{\displaystyle \int_{\Omega} |\nabla u|^2 }{\displaystyle \int_{\Omega} u^2} \,.
\]
If $\Omega$ is a boundary domain (i.e. a domain such that $\partial\Om\cap\partial M\neq\emptyset$), we consider the first eigenvalue of the Laplace-Beltrami operator given by
\begin{equation}\label{lambda1}
\lambda_{\Omega} = \displaystyle \min_{u \in \widetilde H_0^1(\Omega)} \frac{\displaystyle \int_{\Omega} |\nabla u|^2 }{\displaystyle \int_{\Omega} u^2}
\end{equation}
where $\widetilde H_0^1(\Omega)$ denotes the closure of the space $$\{\varphi\in C^\infty(\Om), \;\;\textrm{Supp}(\varphi)\subset \Om\cup\partial M\}$$ for the $H^1$-norm. 
It is very classical that the optimization problem \eqref{lambda1} admits a nonnegative solution if $\Om$ has finite volume, and if $\Om$ is connected such a solution is unique among nonnegative functions whose $L^2$-norm is 1. This function is then called the first eigenfunction of $\Om$.

\medskip

Under smoothness assumption (for example if $\Om$ is a piecewise $C^{1,\alpha}$-domain, see Section \ref{ssect:edge} for more detailed definitions, the space $\widetilde H_0^1(\Omega)$ is equal to the space of functions in $H^1(\Omega)$ with 0 Dirichlet condition on $\partial \Omega \cap \mathring M$, and the function $u$ solving \eqref{lambda1} satisfies:
\begin{equation}\label{P}
\left\{
\begin{array}{rcccl}
	\Delta_{g} \, u + \lambda_{\Om} \, u & = & 0 & \textnormal{in} & \Om \\[1mm]
	u & = & 0 & \textnormal{on} & \partial \Om \cap \mathring M, \\[1mm]
	\displaystyle  g( \nabla  u ,  \nu) & = & 0 & \textnormal{on} & \partial \Om \cap \partial M
\end{array}
\right.
\end{equation}
where $\nu$ denotes the outward normal vector to $\partial M$, which is well-defined as soon as $\Om$ is included in a small enough ball, which will be the case in the whole paper.
This will be referred to as a mixed eigenvalue problem over $\Om$. Moreover, it is also well-known that if there exists $(u,\lambda)$ a nontrivial solution of \eqref{P} for a connected domain $\Om$ such that $u$ is nonnegative, then $\lambda=\lambda_{\Om}$ is the first eigenvalue of $\Om$, and $u$ is the first eigenfunction of $\Om$, up to a multiplicative constant.
\medskip

Let us consider a boundary domain $\Omega_0 \subset M$. $\Omega_0$ is said to be extremal if $\Omega \longmapsto \lambda_\Omega$ is critical at $\Omega_0$ with respect to variations of the domain $\Omega_0$ which preserve its volume. In order to make this notion precise, we first introduce the definition of a {\em  deformation} of $\Omega_0$.
 
\begin{Definition}\label{def:defo}
We say that $( \Omega_t )_{t \in (-t_0, t_0)}$ is a deformation of $\Omega_0$, if there exists a vector field $V$ on $M$, of class $C^2$, such that its flow $\xi$, defined for $t\in(-t_{0},t_{0})$ by 
\[
\frac{d\xi}{dt} (t,p)= V (\xi(t,p)) \qquad \mbox{and} \qquad \xi(0, p) =p \,,
\] 
preserves the boundary of the manifold, i.e. $\xi(t, p) \in \partial M$ for all $(t,p) \in (-t_{0},t_{0})\times\partial M$, and for which 
\[
\Omega_t = \xi (t, \Omega_0).
\] 
The deformation is said to be volume preserving if the volume of $\Omega_t$ does not depend on $t$. 
\end{Definition}
If $(\Omega_t)_{t \in (-t_0, t_0)}$ is a deformation of $\Omega_0$, we denote by $\lambda_t$ the first eigenvalue of the Laplace-Beltrami operator $-\Delta_{g}$ on $\Omega_t$.
We prove in Section \ref{charac} that $t \longmapsto \lambda_t$
is smooth in a neighborhood of $t=0$. If $\Om\subset \mathring{M}$ this fact is standard and follows from the implicit function theorem together with the fact that the first eigenvalue of the Laplace-Beltrami operator is simple, see for example \cite{henrotpierre}. When the boundary $\partial M$ is invariant by the flow of the deformation, as required in Definition \ref{def:defo}, a similar strategy still works when $\partial\Om\cap\partial M\neq\emptyset$, but this is less classical since one needs to manage the singularities of the boundary domains under consideration, see Proposition \ref{lambda}. The derivative at 0 of $t\mapsto\lambda_{t}$ is then called the shape derivative of $\Om\mapsto\lambda_{\Om}$ at $\Om_{0}$ in the direction $V$.

\medskip

This remark allows us to give the definition of an extremal domain.
\begin{Definition}\label{def:extremal}
A domain $\Omega_{0}$ is an \textit{extremal domain} for the first eigenvalue of $-\Delta_{g}$ if for any volume preserving deformation $\{{\Omega}_t\}_t$ of ${\Omega}_{0}$, we have 
\begin{equation}\label{eq:extremal}
\frac{d \lambda_t}{dt} |_{t =0} = 0 \, ,
\end{equation}
where $\lambda_t=\lambda_{\Om_{t}}$ as defined in \eqref{lambda1}.
\label{def:1}
\end{Definition}

All along the paper, we will use a special system of coordinates, that we remind here: let $p \in \partial M$, and let $N$ be the unit normal vector field on $\partial M$ near $p$ that points into $M$. We fix local geodesic normal coordinates $x = (x^{1}, ..., x^n)$ in a neighborhood of $0 \in \mathbb{R}^n$ to parametrize $U_{p}$ a neighborhood of $p$ in $\partial M$ by $\Phi$. We consider the mapping
\begin{equation}\label{eq:normalcoord}
\Psi(x^0,x) = \mbox{Exp}_{\Phi(x)} (x^0 N(\Phi(x)))
\end{equation}
which is a local diffeomorphism from a neighborhood of $0 \in \overline{\mathbb{R}^{n+1}_+}$  (where $\mathbb{R}^{n+1}_+ = \{(x^0,x) \in \mathbb{R}^{n+1}\, :\, x^0 >0 \}$) into $V_{p}$ a neighborhood of $p$ in $M$. For all $\eps >0$ small enough, we denote $ B^+_\eps \subset \mathbb R^{n+1}_+$ the half-ball given by the Euclidean ball of radius $\eps$ centered at the origin and restricted to $x^0 > 0$, and we denote $B^+_{g,\eps}(p)=\Psi( B^+_{\eps}) \subset \mathring{M}$.

\begin{figure}[!ht]
\centering
{\scalebox{.55}{\input{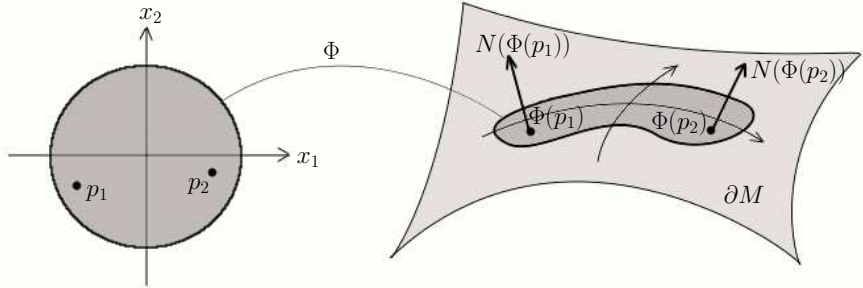}}}
\caption{Our coordinates are defined as $(x^0,x)$, $x$ being the normal geodesic coordinates on $\partial M$ and $x^0$ the coordinate associated to the normal direction.}
\label{fermi}
\end{figure}

\medskip

Now we can state the main result of the paper:
\begin{Theorem}\label{maintheorem}
Assume that $p_{0}\in\partial M$ is a nondegenerate critical point of $\rm{H}$, the mean curvature function of $(\partial M,\tilde{g})$. Then, for all $\eps > 0$ small enough, say $\eps \in (0, \eps_0)$, there exists a boundary domain $\Omega_\eps \subset M$ such that~:
\begin{itemize}
\item[(i)] The volume of $\Omega_\eps$ is equal to the Euclidean volume of $B_\eps^+$. 
\item[(ii)] The domain $\Omega_\eps$ is extremal in the sense of Definition~\ref{def:1}.
\item[(iii)] The boundary $\overline{\partial \Omega_\eps \cap \mathring{M}}$ intersects $\partial M$ orthogonally,
\item[(iv)] The boundary $\partial \Omega_\eps \cap \mathring{M}$ is analytic if $M$ is analytic.
\end{itemize}
Moreover, there exists $c >0$ and, for all $\eps \in (0, \eps_0)$, there exists $p_\eps \in \partial M$ such that $\overline{\partial \Omega_\eps \cap \mathring{M}}$ is a normal graph over $\overline{\partial B^+_{g,\eps} (p_\eps) \cap \mathring{M}}$ for some function $w_\eps$ with 
\[
\|Êw_\eps \|_{C^{2, \alpha}\big( \overline{\partial B^+_{g,\eps} (p_\eps) \cap \mathring{M}}\big)} \leq c \, \eps^3.
\qquad \mbox{and} \qquad {\rm dist} (p_\eps, p_0) \leq c \, \eps \, .
\]
\label{th:1}
\end{Theorem}

This result will be proven in Section \ref{ssect:mainproof}. { The strategy of the proof of this result is inspired by \cite{pacsic}.} In order to give the outline of the paper, we recall here the strategy of the proof and insist on the main differences with \cite{pacsic}. The first step is to characterize the extremality of a domain $\Om_{0}$ with the Euler-Lagrange equation, that leads to:
\begin{equation}\label{eq:ELL}
g(\nabla u,\nu)=\textrm{constant}\;\;\;\;\;\;\textrm{ on }\partial\Om_{0}\cap\mathring{M}.
\end{equation}
The difficulty here is to prove this characterization for domains that are only piecewise smooth (see Section \ref{ssect:edge} where we introduce the notion of boundary edge domain and analyze the regularity theory of mixed boundary value problem in such domains). In particular, we prove in Section \ref{ssect:derivative} that in order to be extremal it is enough for a domain to satisfy \eqref{eq:extremal}  for deformations that preserve the contact angle on $\partial M$; this important fact will be used in the rest of the paper for the construction of extremal domains. This is an interesting difference with the case of critical points of the area functional, as we explain in Section \ref{ssect:isop}: condition \eqref{eq:ELL} contains already the information that the contact angle between $\partial\Om_{0}\cap\mathring{M}$ and $\partial\Om_{0}\cap\partial M$ is constant and equal to $\pi/2$, see Corollary \ref{coro:carac+angle}; this is due to the non-locality of the Euler-Lagrange equation for this problem. It also implies the analytic regularity of $\partial\Om_{0}\cap\mathring{M}$.

\medskip

Then, thanks to a dilation of the metric and a control of the volume constraint, we reformulate in Section \ref{sect:reformulation} the problem into solving for any small $\eps$ the equation
\begin{equation}\label{eq:equation}
F(p,\eps,\bar{v})=0
\end{equation}
where $p\in\partial M$, $\bar{v}\in C^{2,\alpha}(S^n_{+})$ is a function that parametrize a perturbation of the half-geodesic ball $B^+_{g,\eps}(p)$, and $F(p,\eps,\bar{v})$ represents the difference between $g(\nabla u,\nu)$ and its mean value on the boundary of this perturbed half geodesic ball. We then want to solve this equation for $\eps>0$ by using the implicit function Theorem and therefore study the operator $\partial_{\bar{v}}F(p,0,0)$, which is basically related to the second order shape derivative of $\lambda_{1}$ at the Euclidian half-ball. This is the purpose of Sections \ref{ssect:linearization} and \ref{ssect:H}, where we use a symmetrization argument to come down to the study of the same operator in the Euclidian ball, which has been done in \cite{pacsic}. As expected, that operator has a nontrivial kernel (because of the invariance of $\lambda_{1}$ by translation along $\partial\R^{n+1}_{+}$ in the Euclidian setting) and we are only able to solve
$$F(p,\eps,\bar{v}(p,\eps))=k(p,\eps)$$
where $k(p,\eps)$ is a linear function induced by an element of $\partial \R^{n+1}_{+}$, see Proposition \ref{prop:partsol}. Here comes the final step of the proof of Theorem \ref{maintheorem}, which takes into account the geometry of $\partial M$: by studying the expansion of $F, \bar{v}$ with respect to $\eps$, we prove in the end of Section 4 that close to a point $p_{0}$ which is a nondegenerate critical point of the mean curvature of $\partial M$, one can chose $p_{\eps}$ such that $k(p_{\eps},\eps)=0$ and conclude the proof. We insist on the fact that this step is more involved here than in \cite{pacsic}: indeed, the expansions in $\eps$ contain lower order term than in the case without boundary (see Lemma \ref{le:3.33} and Propositions \ref{prop:partsol}, \ref{kernelproj}). Nevertheless, thanks to the choice of our coordinates, the strategy still applies because these lower order terms are orthogonal to linear functions induced by elements of $\partial \R^{n+1}_{+}$. 

\section{Characterization of boundary extremal domains}\label{charac}

\medskip
In this section, we focus on an analytic characterization of extremal domains. The main difficulty here is to handle the shape derivative of $\Om\mapsto\lambda_{\Om}$ in a nonsmooth setting. Indeed, because of the presence of a boundary in $M$, we are naturally led to deal with domains that are only piecewise smooth.
First, we will treat the regularity for the mixed problem \eqref{P} in some domains called boundary edge domains. We compute then the shape derivative of $\Om\mapsto\lambda_{\Om}$ in this setting. Since we have to deal with possibly nonsmooth eigenfunctions, one needs to carefully prove the differentiability of $\Om\mapsto\lambda_{\Om}$ and compute the shape derivative. We will also insist on some important aspects of the non-locality of the extremality condition for $\lambda_{1}$, and compare it with the case of critical points for the area functional.

\subsection{Boundary edge domains and regularity of the eigenfunction}\label{ssect:edge}

\begin{Definition}
Let $\Omega$ be a boundary domain of the manifold $M$, that is to say $\partial\Om\cap\partial M\neq\emptyset$. We say that $\Omega$ is a boundary edge domain if it satisfies the following condition:
\begin{enumerate}
\item $\overline{\partial \Omega \cap \mathring{M}}$ and $\partial \Omega \cap \partial M$ are smooth $n$-dimensional submanifolds with boundary,
\item $\Gamma:=\overline{\partial \Omega \cap \mathring{M}} \cap \partial M$ is a $(n-1)$-dimensional smooth submanifold without boundary. 
\end{enumerate}
In that case, given $p\in\Gamma$ we can define $\omega(p)$ the angle between the normal vector to $\Gamma$ tangent to $\partial M$ and the normal vector to $\Gamma$ tangent to $\partial\Om\cap\mathring{M}$. The function $\omega:\Gamma\to [0,\pi]$  will be referred to as the contact angle of the domain $\Om$, see Fig. \ref{borde}.
\end{Definition}

\begin{figure}[!ht]
\centering
{\scalebox{.5}{\input{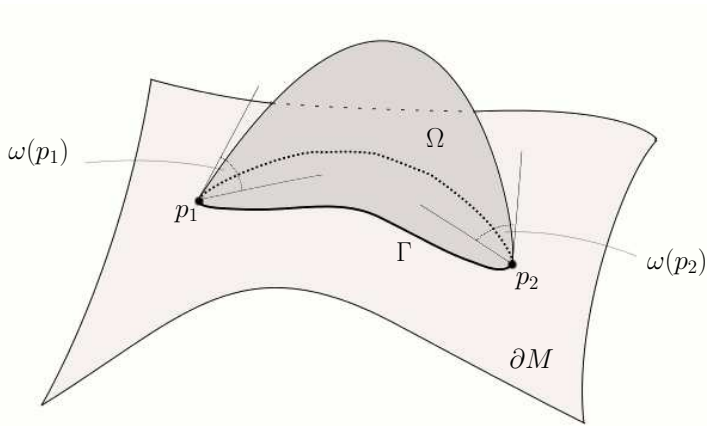}}}
\caption{A boundary edge domain in $M$}
\label{borde}
\end{figure}

\begin{Proposition}\label{prop:H3/2}
Let $\Om$ be a connected boundary edge domain of finite volume such that the contact angle $\omega$ is strictly between 0 and $\pi$. Then there exists $\eps>0$ such that for any $f\in H^{-1/2+\eps}(\Om)$, the solution $u$ of
\begin{equation}\label{eq:mixed}
\left\{
\begin{array}{rcccl}
	-\Delta_{g} \, u & = & f & \textnormal{in} & \Om \\[1mm]
	u & = & 0 & \textnormal{on} & \partial \Om \cap \mathring M, \\[1mm]
	\displaystyle  g( \nabla  u ,  \nu) & = & 0 & \textnormal{on} & \partial \Om \cap \partial M
\end{array}
\right.
\end{equation}
is in the space
$H^{3/2+\eps}(\Om)$.
\end{Proposition}
\begin{Remark}{\rm
It is important for our purpose to work here with Sobolev regularity: if indeed we work with H\"older-regularity, we can only conclude that $u\in C^{0,1/2+\eps}(\overline{\Om})$, which does not suffice to justify the expression of the shape derivative, which uses the trace of the gradient on $\partial\Om$, see Section \ref{ssect:derivative}, while from the fact that $u\in H^{3/2+\eps}(\Om)$, we can deduce that $\nabla u$ has a trace in $L^2(\partial\Om)$ (we use here a trace theorem, valid since under our assumptions, $\Om$ has a Lipschitz boundary).}
\end{Remark}

\begin{proof} 
Let $f\in H^s(\Om)$ where $s\in(-1,0)$. It is well-known from the variational formulation of the problem that there exists a unique $u\in H^1(\Om)$ weak solution of \eqref{eq:mixed}. We wonder for which $s$ we can state that $u\in H^{s+2}(\Om)$.
To that end, we work locally around a point $p\in\Gamma$:
there exist special cylindrical coordinates $(r,\theta,y)$ such that $\Gamma$ correspond to $r=0$, $y\in\Gamma$ parametrizes the edge ($p$ corresponding to $y=0$), and $\Om$ corresponds to $0<\theta<\omega(y)$; since $\Om$ is a boundary edge domain, these coordinates are well-defined and $C^\infty$. From the literature on edge asymptotics, we know that $u$ can be written around $p$ as the sum of a singular function $u_{sing}$ and a remainder term $u_{reg}$ which is more regular that $u_{sing}$; more precisely, it is known (see for example \cite{coda0,coda1,coda2,dauge,dauge2}) that
\begin{eqnarray*}
\textrm{ if }\omega(y)\in (0,\pi/2)&\textrm{then}&u_{sing}(r,\theta,y)=0\;\textrm{ and }u_{reg}\in H^{s+2}(\Om)\\[2mm]
\textrm{ if }\omega(y)\in(\pi/2,\pi)&\textrm{then}&u_{sing}(r,\theta,y)=c(y)\, r^{\pi/2\omega(y)}\,\varphi(\theta,y),\\[2mm]
\textrm{ if }\omega(y)=\pi/2\textrm{ in a neighborhood of }y=0,&\textrm{then}&
u_{sing}(r,\theta,y)=r\left(\sum_{q\geq 1}c_{q}(y)\, \ln^q(r)\, \varphi_{q}(y,\theta)\right),
\end{eqnarray*}
where $c, (c_{q})_{q\in\mathbb{N}}$ (containing only a finite number of non-zero terms) and $\varphi, (\varphi_{q})_{q\in\mathbb{N}}$ are smooth functions (we notice that when $n=1$, the set $\Gamma$ is made of two points, in that case the regularity on $\Gamma$ is an empty condition). Let us conclude in the last two cases. In the second one, we know that
$$\textrm{ if }\alpha>s'-\frac{n+1}{2}, \textrm{ then }r\mapsto r^\alpha\in H^{s'}(\R^{n+1}),$$
and therefore the regularity increases with small angles, and the worst regularity is obtained when the angle is close to $\pi$, but is always strictly better than $H^{3/2}$ which is the limit case when $\omega=\pi$ and $n+1=2$. In the last case, it is clear that $r\ln^q(r)=o(r^{1-\delta})$ for any small $\delta$, so we obtain that the regularity is also better than $H^{3/2}$, therefore there exists $s$ strictly above $-1/2$ such that $u\in H^{s+2}(\Om)$.\\
It remains to understand the case where $\omega(0)=\pi/2$ but $\omega$ is not constant in a neighborhood of $y=0$. In that case, the asymptotic development is more involved (phenomenon of crossing singularities), but it is explained in \cite{coda1,coda2} that up to an arbitrary small loss of regularity, we obtain the same range of regularity as in the case $\omega=\pi/2$, and therefore again $u_{sing}$ is in $H^{3/2+\eps}(\Om)$.
\end{proof}

In the previous proof, we have seen that the regularity is more or less monotone with respect to the contact angle: smaller is the angle, higher is the regularity, and for angles close to $\pi$, the regularity decreases up to the space $H^{3/2}$. However, it is also known that there exists some exceptional angles, for which the regularity is higher than expected (see for example \cite{AK} for a description of this phenomenon for the angle $\pi/4$ in dimension 2). We prove here that the angle $\pi/2$ is such an exceptional angle in our situation. 
More precisely we prove that when the angle is $\pi/2$ everywhere on the interface, the regularity is actually $C^{2,\alpha}$, whereas it was expected to be $C^{0,\alpha}$ for every $\alpha$ in the proof of the previous statement. This will be very useful in the proof of Theorem \ref{maintheorem}. This result is related to the fact that one can use a symmetrization argument to conclude that the first expected term in the asymptotic development of $u$ vanishes.
\begin{Proposition}\label{reg}
Let $\Om$ be a boundary edge domain, such that the angle $\omega$ defined on $\Gamma$ is constant and equal to $\pi/2$. Then
for every $\alpha<1$ and any $f\in C^{0,\alpha}(\overline{\Om})$, the solution $u$ of \eqref{eq:mixed} is in $C^{2,\alpha}(\overline{\Om})$.
\end{Proposition}
\begin{proof}
We use the same setting as in the proof of Proposition \ref{prop:H3/2}, but now in the class of H\"older spaces, so we consider $f\in C^{0,\alpha}(\overline{\Om})$. Around $p\in\Gamma$, from \cite{coda1,coda2,dauge,dauge2}, we know that the exponents in the asymptotic development for the mixed boundary problem are $(\pi/2\omega + k\pi/\omega)_{k\in\mathbb{N}}$, so for the angle $\pi/2$ the first terms are $1$ and $3$ and since $r\mapsto r^3\ln^q(r)$ belongs to the space $C^{2,\alpha}(\overline{\Om})$ for every $\alpha$ and any integer $q$, we conclude that
\begin{equation}\label{eq:dvp}
u(r,\theta,y)=r\left(\sum_{q\geq 1}c_{q}(y)\, \ln^q(r)\,\varphi_{q}(y,\theta)\right)+u_{reg}(r,\theta,y),
\end{equation}
for $y$ close to $0$, $r$ small, $\theta\in(0,\pi/2)$ and
where functions $(c_{q},\varphi_{q})$ are smooth and $u_{reg}$ is in $C^{2,\alpha}$ locally around $p$.

\medskip

The result will be proven if we prove that $c_{q}=0$ for $q\geq 1$. To that end, we use a symmetrization procedure through $\partial M$, using around $p\in\Gamma$ the coordinates $(x^{0},x)$ described in \eqref{eq:normalcoord}. 
We define $U=\Psi^{-1}(\Om\cap B^+_{g,r_{0}}(p))\subset B^+_{r_{0}}$, so that $\partial U\cap(\{0\}\times\R^n)=\Psi^{-1}(\partial\Om \cap \partial M\cap \overline{B^+_{g,r_{0}}(p)})$.
With this choice of coordinates, $U$ is again a boundary edge domain whose contact angle is constant and equal to $\pi/2$ on $\gamma=\Psi^{-1}(\Gamma)$.

We now define $W=\{(x^0,x)\;/(|x^0|,x)\in U\}$ and

$$\forall (x^0,x)\in W, \;\;{\mathring{u}}(x^0,x)=\left\{\begin{array}{ll}
u(x^0,x)&\textrm{ if }x^{0}>0\\
u(-x^0,x)&\textrm{ if }x^{0}<0
\end{array}\right.
\;\;\;\;\textrm{ and similarly we define }\mathring{g}\textrm{ and }\mathring{f}.
$$
Since the contact angle is $\pi/2$, the symmetrized domain $W$ is smooth around 0; using that $u$ satisfies a Neumann boundary condition on $\partial\Om\cap \partial M$, we deduce that $\mathring{u}$ satisfies
$$
\left\{
\begin{array}{rcccl}
	-\Delta_{\mathring{g}}\, \mathring{u}& = & \mathring{f} & \textnormal{in} & W \\[1mm]
	\mathring{u} & = & 0 & \textnormal{on} & \partial W\cap B_{r_{0}}.
\end{array}
\right.
$$
and finally, the symmetrized metric $\mathring{g}$ is no longer $C^\infty$ but has Lipschitz coefficients, and $\mathring{f}$ is again in $C^{0,\alpha}(\overline W)$.
Since the Laplace operator can be written in a divergence form 
$$\Delta_{\mathring{g}}{u}=\frac{1}{\sqrt{|\mathring{g}|}}\partial_{i}\left(\sqrt{|\mathring{g}|}\mathring{g}^{ij}\partial_{j}\mathring{u}\right)$$
we can apply the regularity theory for elliptic PDE in divergence form in a smooth set, with Lipschitz coefficients: precisely, from \cite[Theorem 8.34]{GT} we know that $\mathring u\in C^{1,\alpha} \left(\overline{W}\right)$ and therefore $(c_{q})_{q\geq 1}$ must be zero, and finally $u\in C^{2,\alpha}(\overline{\Om})$.
\end{proof}

\subsection{Shape derivative in nonsmooth domains}\label{ssect:derivative}

\begin{Proposition}\label{lambda}
Let $\Om_{0}$ be a connected boundary domain of finite volume. Assume that $({\Omega}_t)_t$ is a deformation of $\Om_{0}$ induced by the vector field $V$, as defined in Definition~\ref{def:1}. Then $t \longmapsto \lambda_t$ is $C^\infty$ around $t =0$. If moreover $\Om_{0}$ is a boundary edge domain such that the contact angle is strictly between $0$ and $\pi$, then $g(\nabla  u_0 ,  \nu_0)\in L^2(\partial\Om_{0})$ and
\begin{equation}\label{eq:lambda'}
\frac{d\lambda_t}{dt} |_{t =0}= - \int_{\partial \Omega_{0} \cap \mathring{M}}  \left( g(\nabla  u_0 ,  \nu_0) \right)^{2}\  g(V, \nu_0) \, \mbox{\rm dvol}_g,
\end{equation}
where $\mbox{\rm dvol}_g$ is the volume element on $\partial \Omega_{0} \cap \mathring{M}$ for the metric induced by $g$ and $\nu_0$ is the normal vector field on $\partial \Omega_0 \cap \mathring{M}$.
\end{Proposition}
Before proving this result, we give some remarks and consequences.
The differentiability of some similar shape functional for mixed boundary value problem is studied in \cite[Section 3.9]{sokozolesio} in the case of a smooth domain, which corresponds to the case of a angle constant and equal to $\pi$. In that case formula \eqref{eq:lambda'} is not valid since the eigenfunction $u$ is not smooth enough. Also in \cite{bochniaksandig}, the case of angles different from $\pi$ is considered, but for a different shape functional, and restricted to the two-dimensional case. 

\medskip

Proposition \ref{lambda} allows us to characterize extremal domains for the first eigenvalue of the Laplace-Beltrami operator under 0 mixed boundary conditions, and state the problem of finding extremal domains into the solvability of an over-determined elliptic problem. As a consequence of the previous result, we obtain indeed:
\begin{Corollary}\label{coro:carac+angle}
Let  $\Omega_0$ be a boundary edge domain. Then  $\Om_{0}$ is extremal if and only if  the first eigenfunction $u_0$ of $\Om_{0}$ satisfies
\begin{equation}\label{eq:EL}
g (\nabla  u_0 , \nu_0)  =  \textrm{constant} \;\; \textnormal{on}\;\;  \partial \Omega_0 \cap  \mathring{M}
\end{equation}
where $\nu_0$ is the outward normal vector field on $\partial \Omega_0\cap \mathring{M}$.
In that case, $\partial\Omega_0\cap\mathring{M}$ necessarily meets $\partial M$ orthogonally, that is to say the contact angle function $\omega$ is equal to $\pi/2$ on $\Gamma$.
\end{Corollary}

\noindent{\bf Proof of Corollary \ref{coro:carac+angle}:}
Let $\Omega_0$ be a boundary extremal domain for the first eigenvalue of the Laplace-Beltrami operator, with 0 Dirichlet boundary condition on $\partial \Omega_0 \cap \mathring{M}$ and 0 Neumann boundary condition on $\partial \Omega_0 \cap \partial M$. Using Proposition \ref{lambda}, we obtain 
\[
\int_{\partial \Omega_{0} \cap \mathring{M}} \left(g(\nabla  u_0 , \nu_0)  \right)^{2}\  g(V, \nu_0) \ \textnormal{dvol}_g = 0
\]
for all field $V$ preserving the volume of the domain, i.e. such that 
\begin{equation}\label{xixi}
\int_{\partial \Omega_{0} \cap \mathring{M}}  g(V, \nu_0) \ \textnormal{dvol}_g = 0.
\end{equation}
This means that $g(\nabla  u_0 , \nu_0)$ is constant. On the other hand, if $g(\nabla  u_0 , \nu_0)$ is constant, by the previous proposition we have that $\Omega_0$ is extremal, because $V$ satisfy \eqref{xixi}.

\medskip

It remains to investigate the angle between $\partial\Om_{0}\cap \mathring{M}$ and $\partial\Om_{0}\cap\partial M$, when \eqref{eq:EL} is satisfied. Let's assume that $y\mapsto\omega(y)$ is not constantly equal to $\pi/2$; then there exists a  neighborhood in $\mathcal{Y}\subset\Gamma=\overline{\partial\Om\cap \mathring{M}}\cap\partial M$ where $\omega$ is different from $\pi/2$. We work locally around a point $y_{0}\in\mathcal{Y}$.
We need now a more explicit version of the asymptotic development written in the proof of Proposition \ref{prop:H3/2}. To that end, we use the results of \cite{CoDaEdge, CoDaEdge2,coda2} which asserts that since the principal part of our operator is the Euclidian Laplacian, we have, up to a smooth change of coordinates, that
$u_{0}(r,\theta,y)$ can be written $u_{reg}(r,\theta,y)+u_{sing}(r,\theta,y)$ with:
\begin{multline*}
\textrm{ if }\omega(y)\in (0,\pi/2)\textrm{ in }\mathcal{Y}\textrm{, then }\;u_{sing}=0\\
\textrm{ and }u_{reg}\in H^{s+2}(\Om)\textrm{ is flat at order 2, which means }\; u_{reg}=\mathcal{O}(r^2)\textrm{ and }\nabla u_{reg}=\mathcal{O}(r),
\end{multline*}
\vspace{-0.9cm}\begin{multline*}
\textrm{ if }\omega(y)\in(\pi/2,\pi)\textrm{ in }\mathcal{Y}\textrm{, then }\;u_{sing}(r,\theta,y)=c(y)r^{\pi/2\omega(y)}\cos\left(\frac{\pi}{2\omega(y)}\theta\right),\\
\textrm{ and }u_{reg}\textrm{ is more flat than  }u_{sing}\textrm{, meaning }\; u_{reg}=o(r)\textrm{ and }\nabla u_{reg}=o(1),
\end{multline*}
(note that here, with the terminology of \cite{coda1,coda2}, there is no crossing singularities, since $\omega(y)\neq \pi/2$ on $\mathcal{Y}$ and we are only interested in the first term of the asymptotic). Therefore
in the first case $g(\nabla  u_0 , \nu_0)=\mathcal{O}(r)$ and in the second case $g(\nabla  u_0 , \nu_0)$ behaves like $-\frac{\pi}{2\omega(y)}c(y)r^{\pi/2\omega(y)-1}\sin\left(\frac{\pi}{2\omega(y)}\theta\right)$,
 and therefore, in both cases, cannot be a nonzero constant on $\partial\Om\cap \mathring{M}=\{\theta=\omega(y)\}$. This is a contradiction (remind that from maximum principle, the constant $g(\nabla  u_0 , \nu_0)$ cannot be a zero), and one concludes that $\omega(y)=\pi/2$ everywhere on $\Gamma$.\qed
\medskip
\noindent{\bf Proof of Proposition \ref{lambda}:} 
Let $\Om_{0}$ be a boundary domain, connected and of finite volume.
We denote by $\xi_{t}=\xi(t,\cdot)$ the flow associated to $V$, $\nu_t$ the outward unit normal vector field to $\partial \Omega_{t}$. We first remind that, since $\Om_{t}$ is connected, for $t$ small enough $\lambda_{t}$ the first eigenvalue of $\Om_{t}$ with mixed boundary condition is simple, so one can define $t\mapsto u_t  \in \widetilde{H}^1_{0}(\Omega_t)$ the one-parameter family of first eigenfunctions of the Laplace-Beltrami operator, normalized to be positive and have $L^2(\Omega_t)$-norm equal to $1$.
As usual in the computation of a shape derivative, we consider $\widehat{u_{t}}=u_{t}\circ \xi(t,\cdot)\in \widetilde{H}^1_{0}(\Om_{0})$.

\medskip

{\bf Step 1: $\exists\; t_{0}>0$ such that $t\in(-t_{0},t_{0})\mapsto (\widehat{u_{t}},\lambda_{t})\in \widetilde H_0^1(\Omega_{0})\times\R$ is $C^\infty$.} \\
The variational formulation of the equation satisfied by $u_{t}$ is:
$$
\int_{\Om_{t}}g(\nabla u_{t},\nabla \varphi)=\lambda_{t}\int_{\Om_{t}}u_{t}\varphi\, , \;\;\;\forall \varphi\in \widetilde{H}^1_{0}(\Om_{t}).$$
We are going to transport that formulation on the fixed domain $\Om_{0}$, in order to obtain the variational formulation satisfied by $\widehat{u_{t}}\in \widetilde H_0^1(\Omega)$. To that aim, we use the following equality, which relies on the fact that $$\xi_{t}(\partial\Om_{0}\cap\partial M)=\partial\Om_{t}\cap\partial M$$ and is a consequence of the hypothesis $\xi_{t}(\partial M)\subset \partial M$:
$$\widetilde{H}^1_{0}(\Om_{0})=\{\varphi\circ\xi_{t}, \;\varphi\in\widetilde{H}^1_{0}(\Om_{t})\}.$$
With this equality and a change of variable (see for example \cite{henrotpierre} for details), we obtain:
$$
\int_{\Om_{0}}g(A(t)\, \nabla \widehat{u_{t}}\,,\,\nabla \varphi)=\lambda_{t}\int_{\Om_{0}}\widehat{u_{t}}\,\varphi \, J_{t}\,\,\,, \;\;\;\forall \varphi\in \widetilde{H}^1_{0}(\Om_{0}),$$
where 
$$J_{t}=\det(D\xi_{t}), \;\;\textrm{ and }A(t):=J_{t}\, D\xi_{t}^{-1}(D\xi_{t}^{-1})^T.$$
We then define
$$\begin{array}{cccl}
G:&(-t_{0},t_{0})\times \widetilde{H}^1_{0}(\Om_{0})\times \R&\longrightarrow&\widetilde{H}^1_{0}(\Om_{0})'\times\R\\[2mm]
&(t,v,\mu)&\longmapsto&\displaystyle{\left(-{\div_{g}(A(t)\nabla v)}-\mu v J_{t}\, ,\, \int_{\Om_{0}}v^2J_{t}-1\right)}
\end{array}$$
where $\widetilde{H}^1_{0}(\Om_{0})'$ is the dual space of $\widetilde{H}^1_{0}(\Om_{0})$, 
and $-\div_{g}(A(t)\nabla v)$ has to be understood in the weak sense:
$$\langle-\div_{g}(A(t)\nabla v),\varphi\rangle_{\widetilde{H}^1_{0}(\Om_{0})'\times \widetilde{H}^1_{0}(\Om_{0})}= \int_{\Om_{0}}g(A(t)\nabla v,\nabla \varphi).$$
It is easy to check that $G$ is $C^\infty$, see again \cite{henrotpierre} for more details. In order to apply the implicit function theorem for the equation $G(t,\widehat{u_{t}},\lambda_{t})=0$, we focus on the differential of $G$ at $(0,u_{0},\lambda_{0})$ with respect to the couple $(v,\mu)$: 
$$\partial_{(v,\mu)}G(0,u_{0},\lambda_{0})(w,\nu)=\left(-\Delta_{g} w-\nu u_{0}-\lambda_{0}w\,,\,2\int_{\Om_{0}}u_{0}w\right), \;\;\forall (w,\nu)\in \widetilde{H}^1_{0}(\Om_{t})\times \R.$$
Because of the Banach isomorphism Theorem, in order to prove to prove that such differential is an isomorphism, it is enough to prove that given $(f,\Lambda)\in \widetilde{H}^1_{0}(\Om_{0})'\times\R$, the equation
$$\left(-\Delta_{g} w-\nu u_{0}-\lambda_{0}w,2\int_{\Om_{0}}u_{0}w\right)=(f,\Lambda)$$
admits a unique solution $(w,\nu)\in \widetilde{H}^1_{0}(\Om_{0})\times\R$.
The operator $-\Delta_{g}-\lambda_{0}\mathbbm{1}$ has a one-dimensional kernel, spanned by $u_{0}$. Therefore $f+\nu u_{0}$ is in the range of $-\Delta_{g}-\lambda_{0}\mathbbm{1}$ if and only if it is orthogonal to $u_{0}$ (in the sense of the duality $\widetilde{H}^1_{0}(\Om_{0})'\times \widetilde{H}^1_{0}(\Om_{0})$). This leads to the unique value $\nu=-\langle f,u_{0}\rangle$. \\
Moreover, one knows that the solutions $w$ of $\left(-\Delta_{g}-\lambda_{0}\mathbbm{1}\right)w=f+\nu u_{0}$ form a one-dimensional affine space $v_{0}+\textrm{Span}(u_{0})$, so $w=v_{0}+\alpha u_{0}$ for some $\alpha\in\R$. The equation $2\int_{\Om_{0}}u_{0}w=\Lambda$ uniquely determines $\alpha$ and so $w$. We can conclude that $\partial_{(v,\mu)}F(0,u_{0},\lambda_{0})$ is an isomorphism, and therefore $t\mapsto (\widehat{u_{t}},\lambda_{t})$ is $C^\infty$.

\medskip

Now and for the rest of the proof, $\Om_{0}$ is assumed to be a boundary edge domain whose contact angle is always strictly  between 0 and $\pi$.

\medskip





{\bf Step 2: Generalized Green formula:} we prove in this step that given $\eps\in (0,1/2)$ and $\Om$ a Lipschitz domain, denoting $H^s(\Delta_{g},\Om):=\left\{\varphi\in H^s(\Om),\;\Delta_{g} \varphi \in L^2(\Om)\right\}$ for $s\in (1/2,3/2)$ we have:
\begin{multline}\forall u\in H^{3/2-\eps}(\Delta_{g},\Om), \forall v\in H^{1/2+\eps}(\Delta_{g},\Om), \\\int_{\Om} \left(v\Delta_{g} u-u\Delta_{g} v\right)=\langle g(\nabla u,\nu_{0}),v\rangle_{H^{-\eps}(\partial\Om)\times H^\eps(\partial\Om)}-\langle u,g(\nabla v,\nu_{0})\rangle_{H^{1-\eps}(\partial\Om)\times H^{-1+\eps}(\partial\Om)}
\end{multline}
When $u,v$ are smooth, this equality is just the classical Green formula. The above generalization is easily obtained by a density argument, using the following result from \cite[Lemma 2 and 3]{costabel}: 
\begin{multline}
H^{3/2-\eps}(\Delta_{g},\Om)=\{\varphi\in H^{1}(\Om),\;\Delta_{g}\varphi\in L^2(\Om)\textrm{ and }\varphi_{|\partial\Om}\in H^{1-\eps}(\Om)\},\\
\textrm{ and }H^{1/2+\eps}(\Delta_{g},\Om)=\{\varphi\in H^{\eps}(\Om),\;\Delta_{g}\varphi\in L^2(\Om)\textrm{ and }g(\nabla\varphi,\nu_{0})_{|\partial\Om}\in H^{-1+\eps}(\Om)\}
\end{multline}
and that $C^\infty(\overline{\Om})$ is dense in $H^{3/2-\eps}(\Delta_{g},\Om)$.

\medskip

{\bf Step 3: Computation of $\frac{d}{dt}u_{t}$:}
From $u_{t}=\widehat{u_{t}}\circ\xi_{t}^{-1}$, we obtain that $u'=\frac{d}{dt}_{|t=0} u_{t}$ is well-defined in ${\Om_0}$ and that
\begin{equation}\label{eq:u'}
u'=\widehat{u}'-g(\nabla u,V),
\end{equation}
where $\widehat{u}'=\frac{d}{dt}_{|t=0} \widehat{u}_{t}\in \widetilde{H}^1_{0}(\Om_{0})$, well-defined from Step 1.
Using that $u\in H^{3/2+\eps}(\Om_{0})$ and that $\widehat{u}'\in H^{1}(\Om_{0})$, we know from \eqref{eq:u'} that
$u'\in H^{1/2+\eps}(\Om_{0})$.
We also know that, the domain $\Om_{0}$ being piecewise $C^\infty$, the functions $u$ and $u'$ are locally $C^\infty$ on $\overline{\Om_{0}}\setminus\Gamma$. With these regularities, we can compute the equation and the boundary conditions satisfied by $u'$:
first, we differentiate with respect to $t$ the identity 
\begin{equation}\label{eq:1-3}
\Delta_g \, u_t + \lambda_t \, u_t  =0.
\end{equation}
and evaluate the result at $t = 0$ to obtain
\begin{equation}\label{eq:1-4}
\Delta_{g} u'_0 + \lambda_0 \,  u'_0  = -  \lambda'_{0}  \, u_0 \, \textrm{, in }\Omega_{0}.
\end{equation}

Moreover, using again \eqref{eq:u'}, we obtain that
$$u'=-g(\nabla u, V)\textrm{ on }\partial\Om\cap\mathring M.$$
and since $u_0 = 0$ on $\partial \Omega_0 \cap \mathring{M}$, only the normal component of  $V$ plays a r\^ole in the previous formula. Therefore, we have, again since $\xi(t,\partial\Om_{0}\cap\partial M)=\partial\Om_{t}\cap\partial M$:
\begin{equation}\label{eq:1-2}
u' = -  \,  g(\nabla  u_0 ,  \nu_0) \, g(V, \nu_0), \;\;\textrm{ on }\partial \Omega_0 \cap \mathring{M}
\end{equation}

About the Neumann part of the boundary, we have:
$$\textrm{for all }p \in \partial \Omega_0 \cap \partial M, \;\;g(\nabla u_t (\xi (t , p)), \nu_{t}) = 0.
$$
Since $V$ is tangential on $\partial M$, using the normal geodesic coordinates we have $\nu_{t} = -\partial_{x^0}$ on $\partial \Omega_t \cap \partial M$, and in particular  it does not depend on $t$ and
\begin{equation}\label{aa}
g(\nabla u_t (\xi (t , p)), \nu_{t})=-\partial_{x^0}u_t (\xi (t , p))=0.
\end{equation}
So, differentiating (\ref{aa}) with respect to $t$ and evaluating the result at $t = 0$ we obtain  
\begin{equation}\label{bb}
0 = - \partial_{x^0} \partial_t u_0 - g (\nabla \partial_{x^0} u_{0}, V)= - \partial_{x^0} \partial_t u_0 = g (\nabla \partial_t u_0, \nu_{0})  \, 
\end{equation}
on $\partial \Omega_0 \cap \partial M$, where we used the facts that $\partial_{x^0} u_{0} = 0$ on $\partial \Omega_0 \cap \partial M$ and that $g (V, \nu_{0} )= 0$ in $\partial \Omega_0 \cap \partial M$.

\medskip

{\bf Step 5: Computation of $\frac{d}{dt}_{|t=0}\lambda_{t}$:}
From \eqref{eq:1-4}, multiplying by $u$ and integrating over $\Om$, we obtain, using the generalized Green formula together with the regularity we have proven on $u$ and $u'$:
$$\lambda'_{0}=\int_{\Om}(-\Delta_{g} u'-\lambda_{0}u')u=\int_{\Om}(-\Delta_{g} u-\lambda u)u'+\langle u',g(\nabla u,\nu_{0})\rangle_{H^{-\eps}(\partial\Om)\times H^\eps(\partial\Om)}-\langle u,g(\nabla u',\nu_{0})\rangle_{H^{1-\eps}(\partial\Om)\times H^{-1+\eps}(\partial\Om)}.$$
Since $u=0$ on $\partial\Om\cap \mathring{M}$ and $g(\nabla u',\nu_{0})=0$ on $\partial\Om\cap\partial M$, we have $\langle u,g(\nabla u',\nu_{0})\rangle_{H^{1-\eps}(\partial\Om)\times H^{-1+\eps}(\partial\Om)}=0$.
Finally, since $u$ and $u'$ are smooth enough so that $\left(g(\nabla u,\nu_{0})_{|\partial\Om}, u'_{|\partial\Om}\right)\in L^2(\partial\Om)$, we can write 
$$\langle u',g(\nabla u,\nu_{0})\rangle_{H^{-\eps}(\partial\Om)\times H^\eps(\partial\Om)}=\int_{\partial\Om}u'g(\nabla u,\nu_{0})=-\int_{\partial\Om\cap \mathring{M}}(g(\nabla u,\nu_{0}))^2 g(V, \nu),$$
and we finally obtain
$$\lambda'=-\int_{\partial\Om\cap \mathring{M}}(g(\nabla u,\nu_{0}))^2 g(V, \nu).$$
\qed

\subsection{Extremal domains versus the isoperimetric problem}\label{ssect:isop}

As we said, extremal domains are the critical points of the functional 
\[
\Omega \to \lambda_{\Omega}
\]
under a volume constraint ${\rm Vol}_g \, \Omega = \kappa$. The problem of finding extremal domains for the first eigenvalue of the Laplace-Beltrami operator is considered, by the mathematical community, very close to the isoperimetric problem. 

\medskip

Given a compact Riemannian manifold $M$ and a positive number $\kappa<\mbox{Vol}_g(M)$, where $\mbox{Vol}_g(M)$ denotes the volume of the manifold $M$, the isoperimetric problem consists in studying, among the compact hypersurfaces $\Sigma \subset M$ enclosing a region $\Omega$ of volume $\kappa$, those which minimize the area functional
\[
\Omega \to \mbox{Vol}_{g} \, (\partial \Omega \cap \mathring M)
\]
(note that we do not take in account the area of $\partial \Omega$ coming from the boundary of $M$). The solutions of the isoperimetric problem are (where they are smooth enough) constant mean curvature hypersurfaces and intersect the boundary of the manifold orthogonally (see for example \cite{iso-ros}). In fact, constant mean curvature hypersurfaces intersecting $\partial M$ orthogonally are exactly the critical points of the area functional 
\[
\Omega \to \mbox{Vol}_{g} \, (\partial \Omega \cap \mathring M)
\]
under a volume constraint ${\rm Vol}_g \, \Omega = \kappa$. 

\medskip
In the case of a manifod $M$ without boundary, it is well known that the determination of the isoperimetric profile
\[
I_\kappa : = \inf_{\Omega \subset M \, : \, {\rm Vol}_g \, \Omega = \kappa} \mbox{Vol}_{g} \, \partial \Omega 
\] 
is related to the Faber-Kr\"ahn profile, where one looks for the least value of the first eigenvalue of the Laplace-Beltrami operator amongst domains with prescribed volume
\[
FK_\kappa : = \inf_{\Omega \subset M \, : \, {\rm Vol}_g \, \Omega = \kappa } \lambda_{\Omega} 
\]
(see \cite{chavel}). For this reason it is natural to expect that the solutions to the isoperimetric problem for small volumes are close in some sense to the solutions of the Faber-Kr\"ahn minimization problem. And such closeness can be expected also for the corresponding critical points.

\medskip

The results known up to now about extremal domains underline such expectations. In the case of a manifold without boundary, the constructions of extremal domains in \cite{pacsic, del-sic} are the parallel of the constructions of constant mean curvature topological spheres in a Riemannian manifold $M$ done in \cite{Ye, Pac-Xu}. And in the case of a manifold with boundary, our construction is the parallel of the constructions of constant mean curvature topological half-spheres in a Riemannian manifold $M$ done in \cite{fall} for dimension $3$.

\medskip

Nevertheless, Proposition \ref{lambda} and Corollary \ref{coro:carac+angle} show a very interesting difference between extremal domains and critical points of the area functional, based on the following:

\begin{Remark}{\rm
A significant fact contained in the statement of Proposition \ref{lambda} is that the shape derivative for the first eigenvalue of the Laplace-Beltrami operator with mixed boundary condition in the boundary edge domain $\Om_{0}$ does not contain a singular term supported by the ``corner part'' of the boundary $\partial\Om_0$, as it is the case for the area functional, see \eqref{eq:perimeter}.}
\end{Remark}
In order to understand the consequences of this remark, let's compare the Euler-Lagrange equations of the two problems: criticality for $\lambda_{1}$ is written
\begin{equation}\label{eula1}
\frac{d\lambda_{t}}{dt} |_{t =0}=\int_{\partial \Omega_{0} \cap \mathring{M}} \left(g(\nabla  u_0 , \nu_0)  \right)^{2}\  g(V, \nu_0) \ \textnormal{dvol}_g = 0
\end{equation}
whereas for the area functional we have
\begin{equation}\label{eq:perimeter}
\frac{d}{dt}\mbox{Vol}_{g} \, (\partial \Omega_t \cap \mathring M){|_{t =0}} = \int_{\partial \Omega_0 \cap \mathring M} \mbox{H}_{0} \, g(V,\nu_{0})  +  \int_{\Gamma} g(V, \tau_{0}) = 0 \,,
\end{equation}
where $(\Omega_t)_{t}$ is a volume preserving deformation of $\Omega_0$ given by the vector field $V$, $\mbox{H}_{0}$ is the mean curvature of $\partial \Omega_0\cap \mathring{M}$, $\nu_{0}$ is the normal vector on $\partial \Omega_0\cap \mathring{M}$, and $\tau_{0}$  is the normal vector to $\Gamma$ tangent to {$\partial \Om_{0}\cap\mathring{M}$}. For the area functional, the consequence of \eqref{eq:perimeter} is {that in order to be critical $\Omega_0$ must satisfy}, denoting $\nu_{1}$ the normal vector to $\Gamma$ tangent to $\partial M$:
$$\mbox{H}_{0} \equiv \textrm{constant, \Big[ and }{g( \tau_{0},\nu_{1})=0}\textrm{ or equivalently }\omega\equiv\pi/2\textrm{ on }\Gamma\Big],$$
the first condition being obtained with vector fields $V$ supported in $\mathring{M}$ whereas the second condition is obtained thanks to vector fields $V$ that are supported in a neighborhood of $\Gamma$. For $\lambda_{1}$, only using vector fields $V$ that are supported in $\mathring{M}$ we obtain as a consequence of \eqref{eula1} that in order to be critical $\Omega_0$ must satisfy:
\begin{equation}\label{eq:eula1}
g(\nabla  u_0 , \nu_0)=\textrm{constant}\;\;\;\textrm{ on }\partial\Om_{0}\cap\mathring{M}.
\end{equation}
The fact that the contact angle is $\pi/2$ on $\Gamma$ is already contained in the above equation (see Corollary \ref{coro:carac+angle}), and therefore domains that are critical domains for $\lambda_{1}$ in the sense of Definition \ref{def:extremal} (i.e. for any vector field $V$ tangent on $\partial M$) are the same as critical domains for $\lambda_{1}$ restricted to vector fields supported in $\mathring{M}$, which is not the case for the area functional.

\medskip

In other words, one can easily build surfaces that have a constant mean curvature but intersects the boundary $\partial M$ with an angle different from $\pi/2$ (and therefore are not extremal sets for the relative perimeter under volume constraint), whereas every set satisfying \eqref{eq:eula1} intersects the boundary $\partial M$ with angle equal to $\pi/2$.

\medskip


These properties lie on the fact that the operator given by the mean curvature is local while the Dirichlet to Neumann operator is nonlocal.



\section{Analysis of the problem}\label{sect:reformulation}

\subsection{Notations and formulation of the problem}

\noindent{\bf Euclidean notations}. We define the following notations:
\[
\R^{n+1}_{+}=\{x = (x^0,x') = (x^0, x^1, \ldots, x^n)\in \R^{n+1} : x^0 > 0\}
\]
will be the upper Euclidean half-space,
$${B}_1^{+}={B}_{1}\cap \R^{n+1}_{+}$$ will be the upper Euclidean unit half-ball and $$S^{n}_{+}=\{x\in S^{n} : x^0 > 0\}$$ will be the upper Euclidean unit hemisphere. Given a continuous function $f : \overline{S^{n}_+} \longmapsto (0, \infty)$, we also denote 
\[
B_{f}^+ : =  \left\{  x \in \mathbb R^{n+1}_+  \quad :  \quad 0 < |x|   < f (x/|x|) \right\} \, . 
\]

\medskip 

\noindent{\bf Riemannian notations in $(M,g)$}. Let $p$ a point of $\partial M$. We denote by $E_1,...,E_n$ the orthonormal base of $T_p \, \partial M$ associated to the geodesic normal coordinates $x^1,...,x^n$ in $\partial M$ around $p$. If the point $q \in \partial M$ has coordinates $x' \in \mathbb{R}^n$, we set
\begin{equation}\label{TTtheta}
\Theta(x') : = \sum_{i=1}^n x^i \, E_i \in T_p\, \partial M \, .
\end{equation}
The point $q \in \partial M$ whose geodesic coordinates are given by $x'$ is
\[
q = \Phi(x') = \mbox{Exp}^{\partial M}_p (\Theta (x')) \,.
\]
Given a continuous function $f : \overline{S^{n}_{+}} \longmapsto (0, \infty)$ whose $L^\infty$ norm is small (say less than the cut locus of $p$)  we define 
\[
B^+_{g,f} (p) : =  \left\{ \mbox{Exp}^M_{\Phi(x')} (x^0 N(\Phi(x'))) \qquad : \quad  x \in \mathbb R^{n+1}_{+} \qquad 0 < |x|   < f (x/|x|) \right\} \, . 
\]
The subscript $g$ is meant to remind the reader that this definition depends on the metric.

\medskip

\noindent{\bf Formulation of the problem}. Our aim is to show that, for all  $\eps >0$ small enough, we can find a point $p_\eps\in \partial M$ and a (smooth) function $v = v(p_\eps, \eps) : \overline{S^{n}_+} \longrightarrow \mathbb R$ with 0 Neumann condition at the boundary of $S^{n}_{+}$ such that 
\begin{equation}\label{a1mesure}
{\rm Vol} \, B^{+}_{g,\eps(1+v)}(p) =  \eps^n \, {\rm Vol} \, B_1^{+}
\end{equation}
and the over-determined elliptic problem
\begin{equation}\label{a1}
\left\{
\begin{array}{rcccl}
	\Delta_{g} \, \phi + \lambda \, \phi & = & 0 & \textnormal{in} & B^{+}_{g,\eps(1+v)}(p) \\[1mm]
	\phi & = & 0 & \textnormal{on} & \partial B^{+}_{g,\eps(1+v)}(p) \cap \mathring M \\[1mm]
	\displaystyle  g( \nabla  \phi ,  \nu) & = & 0 & \textnormal{on} & \partial B^{+}_{g,\eps(1+v)}(p) \cap \partial M \\[1mm]
	\displaystyle  g( \nabla  \phi ,  \nu) & = & {\rm constant} & \textnormal{on} & \partial B^{+}_{g,\eps(1+v)}(p) \cap \mathring M
\end{array}
\right.
\end{equation}
has a nontrivial positive solution, where $\nu$ is the normal vector on $\partial B^{+}_{g,\eps(1+v)}(p)$. Notice that the 0 Neumann boundary condition on $v$ is justified by Corollary \ref{coro:carac+angle}. Indeed, the half ball $B^{+}_{g,\eps}(p)$ intersects $\partial M$ orthogonally, and then, since an extremal domain also intersects $\partial M$ orthogonally, the deformation $v$ should satisfy a fortiori a 0 Neumann boundary condition.
\medskip

\subsection{Dilation of the metric}

We follow the strategy of \cite{pacsic}, paying attention to the fact that we are working in a more general situation because our domains are boundary edge domains. Our first aim is to give a sense to the problem when $\eps =0$. Observe that, considering the dilated metric $\bar g : =  \eps^{-2} \, g$, Problem \eqref{a1mesure}-\eqref{a1} is equivalent to finding a point $p\in \partial M$ and a function $v : \overline{S^{n}_+} \longrightarrow \mathbb R$ with 0 Neumann condition at the boundary of $S^{n}_{+}$ such that 
\begin{equation}\label{b1mesure}
{\rm Vol} \, B^{+}_{\bar{g},1+v}(p) =  {\rm Vol} \, B_{1}^+
\end{equation}
and for which the over-determined elliptic problem
\begin{equation}\label{b1}
\left\{
\begin{array}{rclll}
	\Delta_{\bar g} \, \bar \phi + \bar \lambda \, \bar \phi & = & 0 & \textnormal{in} & B^{+}_{\bar{g},1+v}(p) \\[1mm]
	\bar \phi & = & 0 & \textnormal{on} & \partial B^{+}_{\bar{g},1+v}(p) \cap \mathring{M} \\[1mm]
		\displaystyle  \bar g( \nabla  \bar \phi , \bar \nu ) & = & 0 & \textnormal{on} & \partial B^{+}_{\bar{g},1+v}(p) \cap \partial M\\[1mm]
	\displaystyle  \bar g( \nabla  \bar \phi , \bar \nu ) & = & {\rm constant} & \textnormal{on} & \partial B^{+}_{\bar{g},1+v}(p) \cap \mathring{M}
\end{array}
\right.
\end{equation}
has a nontrivial positive solution, where $\bar \nu$ is the normal vector on $\partial B^{+}_{\bar{g},1+v}(p)$.  The relation between the solutions of the two problems is simply given by 
\[
\phi = \eps^{- n/2} \, \bar \phi\;\;\;\;\;
\textnormal{and} 
\;\;\;\;\;
\lambda = \eps^{-2}Ê\, \bar \lambda \, .
\]
Let us define the coordinates $y = (y^0,y') = (y^0, y^1, ..., y^n) \in B^+_1$ by
\[
\bar \Psi ( y) : =\mbox{Exp}_{\bar \Phi (y')}^M \left( \eps \, y^0 \, \bar N(\bar \Phi(y'))\right)
\]
where
\[
\bar \Phi ( y') : =\mbox{Exp}_p^{\partial M} \left( \eps \, \sum_{i=1}^n y^i \, E_i \right)
\]
for $p \in \partial M$, and $\bar N$ is the unit normal vector about $\partial M$ for the metric $\bar g$ pointing into $M$.
Using Proposition \ref{fermi-exp} of the Appendix, in the new coordinates $y$ the metric $\bar g$ can be written as
\begin{equation}\label{p000}
\begin{array}{rcl}
\bar g_{00} & = & 1\\[3mm]
\bar g_{0j} & = & 0\\[3mm]
\bar g_{ij} & = & \displaystyle \delta_{ij}\, +\, 2\, \eps\, g(\nabla_{E_{i}}N,E_{j})\, y^0\, +\, \eps^2\, R_{0i0j} \, (y^{0})^2\,  +\,  \eps^2\, g(\nabla_{E_{i}}N,\nabla_{E_{j}}N)\, (y^0)^2 \\[3mm]
& & +\, 2\, \eps^2\, \sum_{k} R_{k0ij} \, y^{k} \, y^{0}\, +\, \frac{1}{3}\,\eps^2 \, \sum_{k,\ell} \tilde R_{ikjl} \, y^{k} \, y^{\ell}\, +\, \mathcal{O}(\eps^3)
\end{array}
\end{equation}
for $i,j,k,l = 1,...n$, where $R$ and $\tilde R$ are respectively the curvature tensors of $M$ and $\partial M$, and
\begin{eqnarray*}
R_{0i0j} & = & g\big( R(N, E_{i}) \, N ,E_{j}\big)\\
R_{k0ij} & = & g\big( R(E_{k}, N) \, E_i ,E_{j}\big)\\
\tilde R_{ijkl} & = & \tilde g\big( \tilde R(E_{i}, E_{k}) \, E_{j} ,E_{\ell}\big).
\end{eqnarray*}
In the coordinates $y$ and the metric $\bar g$, the problem can be continuously extended for $\eps =0$ and in this case it becomes
\begin{equation}\label{formula-000}
\left\{
\begin{array}{rcccl}
	\Delta \, \bar \phi + \bar \lambda \, \bar \phi & = & 0 & \textnormal{in} & B_{1+v}^+ \\[1mm]
	\bar \phi & = & 0 & \textnormal{on} & \partial B_{1+v}^+ \cap \mathbb R^{n+1}_+\\[1mm]
	\displaystyle  \langle \nabla  \bar \phi , \bar \nu \rangle & = & 0 & \textnormal{on} & \partial B_{1+v}^+ \cap \partial \mathbb R^{n+1}_{+}
\end{array}
\right.
\end{equation}
where $\Delta$ denotes the usual Laplacian in $\mathbb R^{n+1}$ and $\langle \cdot, \cdot\rangle$ the usual scalar product in $\mathbb R^{n+1}$, with the normalization 
\begin{equation}
\int_{B_{1+v}^+} \bar \phi^2 \, =1
\label{noral-000}
\end{equation}
and the volume constraint
\[
{\rm Vol} (B_{1+v}^+) = {\rm Vol}(B_1^+ ).
\]
In particular, when $v=0$ we have 
\begin{equation}
\left\{
\begin{array}{rcccl}
	\Delta \phi_{1} + \lambda_{1}\, \phi_{1} & = & 0 & \textnormal{in} & B_{1}^+ \\[1mm]
	\phi_{1} & = & 0 & \textnormal{on} & \partial B_{1}^+ \cap \mathbb R^{n+1}_{+}\\[1mm]
	\displaystyle \langle \nabla \phi_1 , \nu \rangle & = & 0 & \textnormal{on} & \partial B_{1}^+ \cap \partial \mathbb R^{n+1}_{+}
\end{array}
\right.
\label{eq:11-11}
\end{equation}
where $\lambda_{1}$ is the first eigenvalue of the unit Euclidean ball and $\phi_1$ is the restriction to $B_{1}^+$ of the solution to 
\[
\left\{
\begin{array}{rcccl}
	\Delta \tilde \phi_{1} + \lambda_{1}\, \tilde \phi_{1} & = & 0 & \textnormal{in} & B_{1} \\[1mm]
	\tilde \phi_{1} & = & 0 & \textnormal{on} & \partial B_{1}
\end{array}.
\right.
\]
chosen in order to be positive and have $L^2 (B_{1})$ norm equal to $2$.

\subsection{Volume constraint and differentiability with respect to $(\eps,\bar v)$}\label{section3}

In this section, we deal with the volume condition (which leads to replace the variable $v$ by $\bar v$ subject to the condition of having a zero mean), and prove the differentiability of $(\bar\lambda,\bar\phi)$ with respect to $(\eps,\bar v)$. The result is similar to Proposition 3.2 in \cite{pacsic}, and we use the same strategy, though we have to pay attention to the singularities at the boundary of our domain.
Let us define the space 
\[
C^{2,\alpha}_{m,NC} (\overline{S^n_+}) := \left\{v\in C^{2,\alpha} (\overline{S^n_+}) , \int_{S^{n}_+} \bar v \, =0 \, , \, \partial_N v = 0\, \,\textnormal{on}\, \,\partial S^n_+\right\}\,,
\]
where $\partial_N v = 0$ denotes the 0 Neumann condition at the boundary of $S^n_+$.
\begin{Proposition}
\label{pr:1.2}
Given a point $p \in \partial M$, there exists $\eps_{0} >0$, locally uniform in $p$, such that for all $\eps \in (0, \eps_0)$ and all function $\bar v \in C^{2,\alpha}_{m,NC}  (\overline{S^n_+})  $
such that $\|\bar v\|_{C^{2, \alpha} (\overline{S^{n}_+})}  \leq \eps_0$,
there exists a unique positive function $\bar \phi = \bar \phi (p,\eps, \bar v) \in  C^{2, \alpha} (\overline{B^{+}_{\bar{g},1+v}(p)})$, a constant $\bar \lambda = \bar \lambda (p,\eps, \bar v) \in \mathbb R$ and a constant $v_0 = v_0 (p,\eps, \bar v) \in \mathbb R$ such that
\begin{equation}\label{c1mesure}
{\rm Vol}_{\bar g} (B^{+}_{\bar{g},1+v}(p) ) = {\rm Vol} (B_{1}^+ )
\end{equation}
where $v : =  v_0 + \bar v$ and $\bar \phi$ is a  solution to the problem
\begin{equation}\label{formula}
\left\{
\begin{array}{rcccl}
	\Delta_{\bar g} \, \bar \phi + \bar \lambda \, \bar \phi & = & 0 & \textnormal{in} & B^{+}_{\bar{g},1+v}(p) \\[1mm]
	\bar \phi & = & 0 & \textnormal{on} & \partial B^{+}_{\bar{g},1+v}(p) \cap \mathring M \\[1mm]
		\displaystyle  \bar g( \nabla  \bar \phi , \bar \nu ) & = & 0 & \textnormal{on} & \partial B^{+}_{\bar{g},1+v}(p) \cap \partial M
\end{array}
\right.
\end{equation}
which is normalized by 
\begin{equation}\label{noral}
\int_{B^{+}_{\bar{g},1+v}(p)} \bar \phi^2 \, {\rm dvol}_{\bar g} =1 \,.
\end{equation}
In addition $\bar \phi$, $\bar \lambda$ and $v_0$ depend smoothly on the function $\bar v$ and the parameter $\eps$, can be extended smoothly to $\eps =0$ by \eqref{formula-000}, and in particular $(\bar \phi, \bar \lambda,v_{0}) = (\phi_1,\lambda_1,0)$  when $(\eps,\bar v)=(0,0)$.
\end{Proposition}

\begin{proof} The proof of this result is similar to the proof of Proposition 3.2 in \cite{pacsic}, basically based on the implicit function Theorem. Therefore we only describe the differences from \cite{pacsic}, which are the choice of coordinates and the regularity theory for the Laplace-Beltrami operator in domains with singularities.

\medskip

For the choice of coordinates we use the following coordinates: given $(v_{0},\bar v)\in\R\times C^{2,\alpha}_{m,NC}  (\overline{S^n_+}) $ and $v=v_{0}+\bar v$, we consider the parameterization of $B^{+}_{\bar{g},1+v}(p) = B^+_{g,\eps (1+v)} (p)$ given by
\[
\hat \Psi ( y) : =\mbox{Exp}^{M}_{\hat \Phi(y')}\left( \left(1 + v_0 + \chi ( y) \, \bar v \left(\frac{y}{|y|} \right)  \right) \, y^0\, N(\hat\Phi(y')) \right)
\]
where
\[
\hat \Phi (y') = \mbox{Exp}_p^{\partial M} \left( \left(1 + v_0 + \chi ( y) \, \bar v \left(\frac{y}{|y|} \right)  \right) \, \sum_{i=1}^n y^i \, E_i \right).
\]
Here $y = (y^0,y') \in B_1^+$, $\chi$ is a cutoff function identically equal to $0$ when $|y| \leq 1/2$ and identically equal to $1$ when $|y|Ê\geq 3/4$, introduced to avoid the singularity at the origin of the polar coordinates. In these coordinates the metric 
\begin{equation}\label{metrichat}
\hat g : = \hat \Psi^* \bar g
\end{equation}
can be written as 
\[
\hat g  = (1+ v_0)^2 \, \sum_{i,j} (\delta_{ij} + C^{ij}) \, dy_i \, dy_j  \, ,
\]
where the coefficients $C^{ij}=C^{ij}_{\eps,v} \in {C}^{1, \alpha} (\overline{B_1^+})$ are functions of $y$ depending on $\eps$, $v =v_0+\bar v$ and the first partial derivatives of $v$. It is important here to notice that 
\[
(\eps, v_0, \bar v) \longmapsto C^{ij}_{\eps,v}\in {C}^{1, \alpha} (\overline{B_1^+})
\]
are smooth maps, as in \cite{pacsic}.\\ 
Now for all $\psi \in C^{2, \alpha}Ê(B_1^+)$ such that 
\[
\int_{B_1^+} \psi \, \phi_1 \, =0
\]
we define
\[
N (\eps, \bar v , \psi , v_0) : = \left( \Delta \psi +  \lambda_1 \, \psi  +  (\Delta_{\hat g} - \Delta + \mu )\, (\phi_1 + \psi)  \, , \,   {\rm Vol}_{\hat g}Ê(B_1^+) - {\rm Vol}\, (B_1^+)\right)
\]
where $\mu$ is given by 
\[
\mu  = -  \int_{B_1^+} \phi_1 \,   (\Delta_{\hat {g}} - \Delta ) \, (\phi_1  + \psi) \,,
\]
so that the first entry of $N$ is $L^2(B_1^+)$-orthogonal to $\phi_1$ (for the Euclidean metric). Thanks to the choice of coordinates, the mapping $N$ is a smooth map from a neighborhood of $(0,0,0,0)$ in  $[0, \infty) \times {C}_{m,NC}^{2, \alpha} (\overline{S^{n}_+}) \times {C}^{2, \alpha}_{\perp \,, \, 0} (\overline{B_1^+}) \times \mathbb R$ into a neighborhood of $(0,0)$ in  ${C}_\perp^{0, \alpha} (B_1^+) \times \mathbb R$. Here the subscript $\perp$ indicates that the functions in the corresponding space are $L^2(B_1^+)$-orthogonal to $\phi_1$ and the subscript $0$ indicates that the functions satisfy the mixed condition at the boundary of $B_1^+$. The differential of $N$ with respect to $(\psi,v_0)$, computed at  $(0,0,0,0)$, given by 
\[
\partial_{(\psi,v_0)} N (0,0,0,0) = \left( \Delta+ \lambda_1 \, , \, n \, {\rm Vol}(B_1^+)\right) \,
\]
 is invertible from ${C}_{\perp, 0}^{2, \alpha} (\overline{B_1^+}) \times \mathbb R$ into ${C}_\perp^{0, \alpha} (\overline{B_1^+}) \times \mathbb R$, by Proposition \ref{reg}. Then the implicit function theorem applies as in \cite{pacsic} and completes the proof of the result. 
\end{proof}

\subsection{Strategy for the proof of Theorem \ref{maintheorem}}

We define the operator
\[
F (p, \eps, \bar v) =  \displaystyle \bar g (\nabla \bar \phi ,  \bar \nu )  \, |_{\partial B^{+}_{\bar{g},1+v}(p) \cap \mathring M}   - \frac{1}{{\rm Vol}_{\bar g} \big(\partial B^{+}_{\bar{g},1+v}(p) \cap \mathring M\big)} \, \int_{\partial B^{+}_{\bar{g},1+v}(p) \cap \mathring M}Ê \, \bar g (\nabla \bar \phi ,  \bar \nu ) \, \mbox{dvol}_{\bar g} \, ,
\]
where $\bar \nu$ denotes the unit normal vector field to $\partial B^{+}_{\bar{g},1+v}(p) \cap \mathring M$, {$(\bar \phi, v_0)$} is the solution of \eqref{c1mesure}-(\ref{formula})-\eqref{noral}. Recall that $v = v_0 + \bar v$. The operator $F$ is locally well defined in a neighborhood of $(p,0,0)$ in  $\partial M \times [0,\infty) \times C^{2,\alpha}_{m,NC}(S^n_+)$, and after canonical identification of $\partial B^{+}_{\bar{g},1+v}(p) \cap \mathring M$ with $S^{n}_+$ we can consider that it takes its values in $C^{1,\alpha}(S^n_+)$. Moreover, it is easy to see that the zero mean condition is preserved, and then we will write that $F$ takes its values in $C^{1,\alpha}_{m}(S^n_+)$.\\
{Our aim is to find $(p,\eps,\bar v)$ such that $F(p,\eps,\bar v)=0$. Observe that, with this condition, $\bar \phi=\bar \phi(\eps,\bar v)$ will be the solution to the problem (\ref{b1}).}

\medskip

Following the proof of the previous result, we have the alternative expression for $F$:
\[
F (p, \eps, \bar v) =  \displaystyle \hat g (\nabla \hat \phi ,  \hat \nu )  \, |_{\partial B_{1}^+ \cap \mathbb R^{n+1}_+}   -\frac{1}{ {\rm Vol}_{\hat g} (\partial B_{1}^+ \cap \mathbb R^{n+1}_+) } \,  \int_{\partial B_{1}^+ \cap \mathbb R^{n+1}_+}Ê \, \hat g (\nabla \hat \phi ,  \hat \nu ) \, \mbox{dvol}_{\hat g} \, ,
\]
where this time $\hat \nu$ is the  the unit normal vector field to $\partial B_1^+$ using the metric $\hat g$ defined by \eqref{metrichat}.

\medskip 

Our aim is to solve the equation
\[
F(p,\epsilon,\bar v) = 0\,
\]
for some $(p,\epsilon, \bar v)$. The first question we should consider is the following: if we fix a point $p \in \partial M$, can we find for all $\eps$ small enough a function $\bar v = \bar v (\eps)$ in order that
\[
F(p,\epsilon,\bar  v (\eps)) = 0\, ?
\]
The answer will be negative, because we will see that the kernel $K$ of 
\[
\partial_{\bar v}F(p,0,0) : C^{2, \alpha}_{m,NC} (\overline{S^n_+} ) \to C^{1, \alpha}_{m} (\overline{S^n_+} )
\] 
is nontrivial. Nevertheless, we will obtain a characterization of $K$ proving that it is given by the space of linear functions (restraint to the half-sphere) depending only on the coordinates $y^1, ..., y^n$, i.e. functions
\[
\begin{array}{ccc}
\overline{S^n_+} & \to & \mathbb{R}\\
y &\to &\langle a, y \rangle
\end{array}
\]
for some $a = (a^0,a) \in \R^{n+1}$ with $a^0 =0$. Moreover we will prove that $\partial_{\bar v}F(p,0,0)$ is an isomorphism from $K^\bot$ to the image of $\partial_{\bar v}F(p,0,0) $, and then the implicit function theorem will give the following result: for all $\eps$ small enough there exist an element $k(\eps) \in K$ and a function $\bar v(\eps)$ such that
\[
F(p,\epsilon,\bar  v (\eps)) = k(\eps) \,.
\]
Clearly, since we fixed the point $p$, the function $\bar v$ and the element $k$ depend also on $p$, and in fact  we have to write
\[
F(p,\epsilon,\bar  v(p,\eps) ) = k(p,\eps) \,.
\]
In the last section we will show that it is possible to apply the implicit function theorem to the equation 
\[
k(p,\eps) = 0
\]
obtaining that: for all $\eps$ small enough, there exists a point $p_\eps$ such that 
\[
k(p_\eps, \eps) = 0\,.
\]
and this will complete the proof of the result.

\section{Solving the problem}

\subsection{Computation of the linearization of $F$ with respect to $\bar v$ at $(p,\eps,\bar v) = (p, 0,0)$}\label{ssect:linearization}

In Section \ref{section3} we established the existence of a unique positive function $\bar \phi \in {C}^{2, \alpha} \left(B_{1+v}^+\right)$ (close to $\phi_1$), a constant $\bar \lambda \in \mathbb R$ (close to $\lambda_1$) and a constant $v_0 \in \mathbb R$ (close to $0$), solutions to \eqref{c1mesure}-\eqref{formula}-\eqref{noral}. Recall that  $\lambda_{1}$ is the first eigenvalue of $- \Delta$ in the half ball $B_{1}^+$ with $0$ mixed boundary condition and $\phi_{1}$ is the associated eigenfunction which is normalized to be positive and have $L^2 (B_{1}^+)$  norm equal to $1$. 
\medskip
For all $ \bar v \in C^{2, \alpha}_{m,NC} (\overline{S^{n}_+})$ let $\psi$ be the (unique) solution of 
\begin{equation}\label{c00}
\left\{
\begin{array}{rcll}
	\displaystyle \Delta \psi + \lambda_{1}Ê\,Ê \psi & = & 0  & \textnormal{in} \qquad B_1^+ \\[1mm]
	\psi & = &  - \displaystyle  {\partial_r \phi_{1}} \, \bar v  & \textnormal{on}\qquad \partial B_1^+ \cap \mathbb R^{n+1}_+ \\[1mm]
	\displaystyle  \langle \nabla  \psi , \nu \rangle & = & 0 & \textnormal{on}\qquad \partial B_1^+ \cap \partial \mathbb R^{n+1}_+
\end{array}
\right.
\end{equation}
which is $L^2(B_1^+)$-orthogonal to $\phi_1$. We define
\begin{equation}
L_{0}(\bar v) : = \left( {\partial_r \psi }  + {\partial^2_r \phi_{1}} \, \bar v \right) \, |_{\partial B_1^+ \cap \mathbb R^{n+1}_+}
\label{Hache}
\end{equation}
Clearly we have
\[
L_{0} : C^{2, \alpha}_{m,NC} (\overline{S^{n}_+}) \to C^{1, \alpha}_{m} (\overline{S^{n}_+}) \,.
\]

\begin{Proposition}\label{calcul-deriv}
The linearization the operator $F$ with respect to $\bar v$ computed at $(p,0,0)$, i.e. 
\[
\partial_{\bar v}F(p,0,0)\, ,
\]
is equal to $L_{0}$.
\end{Proposition}

\begin{proof} When $\eps =0$ we have already seen that $\bar g$ in the coordinates $y$ is the Euclidean metric.
If $v \in C^{2,\alpha}_{m} (S^{n})$ we can define the operator $F$:
\[
\tilde F (v) =  \displaystyle \langle \nabla \tilde \phi ,  \tilde \nu \rangle  \, |_{\partial B_{1+v}}   - \frac{1}{{\rm Vol}\big(\partial B_{1+v}\big)} \, \int_{\partial B_{1+v}}Ê \, \langle \nabla \tilde \phi ,  \tilde \nu \rangle \,  \, ,
\]
where $\tilde \nu$ denotes the unit normal vector field to $\partial B_{1+v}$ and {$\tilde \phi$} is the solution, with $L^2$-norm equal to $2$, of 
\begin{equation}\label{vecchio}
\left\{
\begin{array}{rcccl}
	\Delta\, \tilde \phi + \bar \lambda \, \tilde \phi & = & 0 & \textnormal{in} & B_{1+v} \\[1mm]
	\tilde \phi & = & 0 & \textnormal{on} & \partial B_{1+v} 
\end{array}\,.
\right.
\end{equation}
After identification of $\partial B_{1+v}$ with $S^n$ we can considered the operator $\tilde F$ well defined from $C^{2,\alpha}_{m} (S^{n})$ into $C_m^{1,\alpha}(S^{n})$. In the proof of Proposition 4.3 in \cite{pacsic} it is proved that the linearization of $\tilde F$ with respect to $v$ at $v=0$ is given by the operator 
\begin{equation}\label{tildeH}
\begin{array}{ccccc}
	\tilde L_{0} & : & C^{2, \alpha}_m (S^{n}) & \longrightarrow & C^{1, \alpha}_{m} (S^{n})\\ 
	& & v & \mapsto & \left( {\partial_r \tilde \psi }  + {\partial^2_r \tilde \phi_{1}} \, v \right) \, |_{\partial B_1}
\end{array}
\end{equation}
where $\tilde \phi_1$ is the first eigenfunction of $-\Delta$ in $B_1$ with 0 Dirichlet boundary condition and normalized to have $L^2$-norm equal to $2$, and $\tilde \psi$ is the (unique) solution of 
\begin{equation}\label{altrro}
\left\{
\begin{array}{rcll}
	\displaystyle \Delta \tilde \psi + \lambda_{1}Ê\,Ê \tilde \psi & = & 0  & \textnormal{in} \qquad B_1 \\[1mm]
	\tilde \psi & = &  - \displaystyle  {\partial_r \tilde \phi_{1}} \, v  & \textnormal{on}\qquad \partial B_1
\end{array}
\right.
\end{equation}
which is $L^2(B_1)$-orthogonal to $\tilde \phi_1$. Notice that $\phi_1$ and $\psi$ are then the restrictions of $\tilde \phi_1$ and $\tilde \psi$ to the half-ball $B^+_1$.

\medskip

Let $w$ be a function in $ C^{2, \alpha}_{m,NC} (\overline{S^{n}_+})$. We extend the function $w$ to a function $\tilde w$ over all $S^n$ in this way: for $(y^0,y^1,...,y^n) \in S^n_+$ we set
\[
\tilde w(-y^0,y^1,...,y^n) = w(y^0,y^1,...,y^n)\,.
\]
Observe that $\tilde w \in C^{2, \alpha} (S^{n})$ because the function $w$ satisfies the Neumann condition at the boundary of $S^n_+$, and his mean is 0 because are 0 the means over $S^n_+$ and over the complement of $S^n_+$. We conclude that $\tilde w \in C^{2, \alpha}_{m,Sym} (S^{n})$, where the subscript $Sym$ means that the function is symmetric with respect to the hyperplane $\{x_0 = 0\}$, and $m$ means as usual that the function has mean 0. We have defined the mapping 
\begin{equation}\label{alpha}
\begin{array}{ccccc}
	\alpha & : & C^{2, \alpha}_{m,NC} (\overline{S^{n}_+}) & \longrightarrow & C^{2, \alpha}_{m,Sym} (S^{n})\\
	& & w & \mapsto & \tilde w
\end{array} \,,
\end{equation}
and it is easy to see that this mapping in an isomorphism.

\medskip

If we consider the operator $\tilde F$ defined only in $C^{2,\alpha}_{m,Sym} (S^{n})$, it is natural that its linearization with respect to $v$ at $v=0$ is given by the operator $\tilde L_{0}$ restricted to $C^{2,\alpha}_{m,Sym} (S^{n})$ with image in $C^{1,\alpha}_{m,Sym} (S^{n})$. We observe that if $v \in C^{2,\alpha}_{m,Sym} (S^{n})$, then the solution of (\ref{vecchio}) is symmetric with respect to the hyperplain $\{x_0 = 0\}$ and the normal derivative with respect to $x_0$ computed at $\{x_0 = 0\}$ is 0. Then from the definitions of $F$ and $\tilde F$ we conclude that
\[
F(p,0,\bar v) = \tilde F(\alpha(\bar v))|_{\partial B_1^+ \cap \mathbb R^{n+1}_+} 
\]
where $\alpha$ is the isomorphism defined in \eqref{alpha}. We define also the mapping 
\[
\begin{array}{ccccc}
	\beta & : & C^{1, \alpha}_{m,Sym} (S^{n}) & \longrightarrow & C^{1, \alpha}_{m,NC} (S^{n}_+)\\ 
	& & v & \mapsto & v|_{S^{n}_+}
\end{array}
\]
and we observe that it is an isomorphism. We claim that
\[
L_{0} = \beta \circ \tilde L_{0} \circ \alpha.
\]
We remark that the operator $\beta \circ \tilde L_{0} \circ \alpha$ is defined on $C^{2, \alpha}_{m,NC} (\overline{S^{n}_+})$ and his image is contained in $C^{1, \alpha}_{m,NC} (S^{n}_+)$. We have to prove that
\[
L_{0}(w) = \tilde L_{0} (\tilde w) |_{\partial B_1^+ \cap \mathbb R^{n+1}_+}
\]
By the symmetry of the funcion $\tilde w$ with respect to the hyperplane $\{x_0 = 0\}$, we conclude that the solution of \eqref{altrro} with $v = \tilde w$ is symmetric with respect to the hyperplane $\{x_0 = 0\}$, then $\partial_{x_0} \tilde \psi \, |_{\{x_0=0\}}= 0$ and $\tilde L_{0} (\tilde w)$ is symmetric with respect to the hyperplane $\{x_0 = 0\}$. So the restriction of $\tilde \psi$ to the half-ball $B_1^+$ is the solution of (\ref{c00}), where $\bar v = w$, and $L_{0}(w)$ is exactly the restriction of $\tilde L_{0} (\tilde w)$ to $\partial B_1^+ \cap \mathbb R^{n+1}_+$. This completes the proof of the claim.
Using this relation we conclude that that
\[
L_0 (\bar w) = \tilde L_{0}(\alpha(\bar w))|_{\partial B_1^+ \cap \mathbb R^{n+1}_+} \,.
\]
This completes the proof of the proposition.\end{proof}
\subsection{Study of the operator $L_{0}$}\label{ssect:H}

{\begin{Proposition}
\label{H}
The operator 
\[
L_{0} : C^{2, \alpha}_{m,NC} (\overline{S^{n}_+}) \longrightarrow C^{1, \alpha}_{m,NC} (S^{n}_+) , 
\]
is a self adjoint, first order elliptic operator. Its kernel $K$ is given by the space of linear functions depending only on the coordinates $y^1, ..., y^n$, i.e. functions
\[
\begin{array}{ccc}
\overline{S^n_+} & \to & \mathbb{R}\\
y &\to &\langle a, y \rangle
\end{array}
\]
for some $a = (a^0,a') \in \R^{n+1}$ with $a^0 =0$. Moreover, $L_{0}$ has closed range and is an isomorphism from $K^\bot$ to $Im(L_0)$, where $K^\bot$ is the space $L^2$-orthogonal to $K$ in $C^{2, \alpha}_{m,NC} (\overline{S^{n}_+})$ and $Im(L_0)$ denotes the range of $L_0$ in $C^{1, \alpha}_{m,NC} (S^{n}_+)$.
\end{Proposition}}
\medskip

\begin{proof} Let $\tilde L_{0}$ the operator defined in \eqref{tildeH} and $\alpha$ the isomorphism defined in \eqref{alpha}. In Proposition 4.2 of \cite{pacsic} it is proved that:
\begin{itemize}
\item $\tilde L_{0}$ is a self adjoint, first order elliptic operator,
\item its kernel is given by the space of linear functions restraint to $S^{n}$, and
\item there exists a constant $c >0$ such that 
\begin{equation}\label{norme}
\| v \|_{C^{2, \alpha}(S^{n})}  \leq c \, \| \tilde L_{0}(v)  \|_{C^{1, \alpha}(S^{n})} \, ,
\end{equation}
provided that $v$ is $L^2(S^{n})$-orthogonal to the kernel of $\tilde L_{0}$.
\end{itemize}
{The last elliptic estimate implies that the operator $\tilde L_0$ has closed range, and using the other two properties we have that $\tilde L_{0}$ is an isomorphism from the space $L^2$-orthogonal to its kernel and its range.}

\medskip

We are interested in considering the operator $\tilde L_{0}$ defined only in the domain $C^{2, \alpha}_{m,Sym} (S^{n})$ and from now on $\tilde L_{0}$ will be defined only in $C^{2, \alpha}_{m,Sym} (S^{n})$. The image of $\tilde L_{0}$ is naturally given by functions that are symmetric with respect to the hyperplane $\{x_0 = 0\}$, then we have
\[
\tilde L_{0} : C^{2, \alpha}_{m,Sym} (S^{n}) \longrightarrow C^{1, \alpha}_{m,Sym} (S^{n})
\]
We can conclude that the new operator $\tilde L_{0}$ is a self-adjoint, first order elliptic operator, with kernel  $\tilde K$ given by the space of linear functions which are symmetric with respect to the hyperplane $\{x_0 = 0\}$, i.e. functions
\[
\begin{array}{ccc}
\overline{S^n_+} & \to & \mathbb{R}\\
y &\to &\langle a, y \rangle
\end{array}
\]
for some $a = (a^0,a') \in \R^{n+1}$ with $a^0 =0$. Inequality (\ref{norme}) holds naturally also for the new operator $\tilde L_{0}$, provided $v$ is $L^2(S^{n})$-orthogonal to $\tilde K$.
\medskip

From the proof of Proposition \ref{calcul-deriv} we have
\[
L_{0} = \beta \circ \tilde L_{0} \circ \alpha \,.
\]
With this caracterization of the operator $L_{0}$ and the properties of $\tilde L_{0}$, we deduce that the kernel  of $L_{0}$ is given by the space $K$ of functions 
\[
\begin{array}{ccc}
\overline{S^n_+} & \to & \mathbb{R}\\
y &\to &\langle a, y \rangle
\end{array}
\]
for some $a = (a^0,a') \in \R^{n+1}$ with $a^0 =0$, and that {$L_{0}$ has closed range and is an isomorphism from $K^\bot$ to $Im(L_0)$, where $K^\bot$ is the space $L^2$-orthogonal to $K$ in $C^{2, \alpha}_{m,NC} (\overline{S^{n}_+})$ and $Im(L_0)$ denotes the range of $L_0$ in $C^{1, \alpha}_{m,NC} (S^{n}_+)$.}
\end{proof}

\subsection{Solving the problem on the space orthogonal to the kernel of $L_{0}$}

\begin{Lemma}
\label{le:3.33}
Let $p \in \partial M$. There exists a function $f_p \in C^{1, \alpha}([0,1])$ such that
\[
F(p, \eps, 0 ) (y^0,y') = \eps\, f_p(y^0)\, +  {\mathcal O}(\eps^2)
\]
for all $\eps$ small enough.
\end{Lemma}

\begin{proof} We keep the notations of the proof of the Proposition \ref{pr:1.2} with $\bar v \equiv 0$. Since $\bar v \equiv 0$, we have  
\[
N (\eps, 0, 0, 0)  = \left( (\Delta_{\hat g} - \Delta + \mu  )\, \phi_1   \, , \,   {\rm Vol}_{\hat g}Ê(B_{1}^+) - {\rm Vol} \, (B_{1}^+) \right) \, ,
\]
and 
\[
\mu  = -  \int_{B_{1}^+} \phi_1 \,   (\Delta_{\hat {g}} - \Delta ) \, \phi_1  \,  \, .
\]
If in addition  $v_0 =0$, we can estimate
\[
\hat{g}_{ij}  = \delta_{ij}  + \hat G_{ij}\,\eps\,y^0 +  {\mathcal O}(\eps^2)\, ,
\]
where $\hat G_{ij}$ are real constants. Hence, by the symmetry of the problem,
\[
N(\eps, 0, 0, 0) (y^0,y') =  \eps \, (\varphi(y^0, |y'|), V) +  {\mathcal O}(\eps^2) \, ,
\]
where the first component of $\varphi \in C^{0,\alpha}([0,1]^2)$ and $V$ is a real number.
The implicit function theorem immediately implies that the solution of 
\[
N(\eps, v_0, 0,  \psi) =0
\]
satisfies
\[
\| \psi (\eps, p, 0) \|_{C^{2, \alpha}}Ê+ |v_0 (\eps, p, 0)|Ê\leq c\, \eps
\]
but in addition there exist a function $\tilde \psi_p \in C^{2, \alpha}([0,1]^2)$ such that
\[
\psi (\eps, p, 0) (y^0,y') = \eps \tilde \psi_p(y^0,|y'|) +  {\mathcal O}(\eps^2)\,.
\]
To complete the proof, observe that $\hat \nu = (1+ v_0)^{-1} \, \partial_r$ on $\partial B_{1}^+ \cap \mathbb R^{n+1}_+$ when $\bar v\equiv 0$. Therefore there exist a function $\hat f_p \in C^{2, \alpha}([0,1]^2)$ such that
\[
{\hat g (\nabla \hat \phi, \hat \nu )} (y^0,y') = \partial_r \phi_1 + \eps \, \hat f_p(y^0, |y'|) +  {\mathcal O}(\eps^2)\,.
\]
(be careful that  $\hat g$ is defined with $v_0 = v_0(\eps, p, 0)$ and $\bar v \equiv 0$). Since $\partial_r \phi_1$ is constant along $\partial B_{1}^+ \cap \mathbb R^{n+1}_+$, we conclude that there exist a function $f_p \in C^{2, \alpha}([0,1])$ such that
\[
F(p, \eps, 0 ) (y^0,y') = \eps f_p(y^0) +  {\mathcal O}(\eps^2)\,.
\]
This completes the proof of the Lemma.
\end{proof}

\begin{Proposition}\label{prop:partsol}
There exists $\eps_0 >0$ such that, for all $\eps \in [0, \eps_0]$ and for all $p$ in a compact subset of $\partial M$, there exists a unique function $\bar v = \bar v(p, \eps) \in K^\bot$ such that 
\[
F (p, \eps, \bar v (p,\eps) )  \in K \,.
\]
The function $\bar v (p, \eps)$ 
depends smoothly on $p$ and $\eps$ and 
\[
\bar v (p,\eps) (y^0,y') = \eps \tilde v_p(y^0) + \mathcal O (\eps^2)
\]
for a suitable function $\tilde v_p \in C^{2,\alpha}([0,1])$.
\end{Proposition}
\begin{proof} We fix $p$ in a compact subset of $\partial M$ and define
\[
\bar F (p, \eps, \bar v , a) : = F (p, \eps, \bar v )  + \langle a, \cdot \rangle 
\]
By Proposition \ref{pr:1.2}, $\bar F$ is a $C^1$ map from a neighborhood of ${(p,0,0,0)}$ in $M \times [0, \infty) \times K^\bot \times \partial \mathbb R^{n+1}_+$ into a neighborhood of $0$ in $C^{1, \alpha} (S^{n}_+)$. Moreover we have 
\begin{itemize}
\item $\bar F (p, 0, 0 , 0) =0$, 
\item the differential of $\bar F$ with respect to $\bar v$ computed at ${(p,0,0,0)}$ is given by $L_{0}$ restricted to $K^\bot$, and
\item the image of the linear map $a \longmapsto \langle a , \cdot \rangle$, $a = (a_0,a')$ with $a_0 = 0$ coincides with $K$.
\end{itemize} 
Thanks to the result of Proposition~\ref{H}, the implicit function theorem can be applied to the equation
\[
\bar F (p, \eps, \bar v , a) = 0
\]
at $(p,0,0,0)$ with respect to the variable $\eps$. We obtain the existence of $\bar v(p,\eps) \in C^{2, \alpha}_{m,NC} (S^{n}_+)$ and $a(p, \eps) \in \partial \mathbb R^{n+1}_+$, smoothly depending on $\eps$ such that 
\[
\bar F (p, \eps, \bar v(p,\eps) , a(p, \eps)) = 0 \,,
\]
that means, by the definition of $\bar F$,
\[
F (p, \eps, \bar v (p,\eps))  \in K \,.
\]
The fact that $\bar v$ depends smoothly on $p$ and $\eps$ is standard. The $\eps$-expansion of $\bar v$ follow at once from Lemma~\ref{le:3.33}.
\end{proof}

\medskip

\subsection{Projecting over the kernel of $L_{0}$: appearance of the mean curvature of $\partial M$}

Thanks to Proposition \ref{prop:partsol} we are able to build, for all $p$ in a compact subset of $\partial M$ and $\eps$ small enough, a function $\bar v(p,Ê\eps)$ in $K^\bot$ such that 
\[
F(p, \eps, \bar v(p, \eps) ) \in K \,.
\]
Now, as natural, we project the operator $F$ over its $K$ and we then we have to find, for each $\eps$, the good point $p_\eps$ in order that such the projection of $F$ over $K$ is equal to 0. In other words, for all $\eps$ small enough we want to find a point $p_\eps \in \partial M$ such that
\[
\int_{S^n_+} F(p_\eps, \eps, \bar v(p_\eps, \eps) ) \, \langle b, \cdot \rangle\,  = 0
\] 
for all $b \in \partial \mathbb{R}^{n+1}_+$. The main result of this section is the following: 

\begin{Proposition}\label{kernelproj}
For all $p \in \partial M$ and all $b = (0,b') \in \partial \mathbb{R}^{n+1}_+$ with $|b|=1$, we have the following $\eps$-expansion:
\[
\int_{S^{n}_+} F(p,\eps,\bar v(p,\eps)) \,  \langle b, \cdot \rangle  =  C \,\eps^2\, \tilde g( \nabla^{\tilde g} {\rm H} (p) , \Theta (b') )  + \mathcal{O}(\eps^3) \,.
\]
where $C$ is a real constant, ${\rm H}$ is the mean curvature of $\partial M$, $\tilde g$ is the metric of $\partial M$ induced by $g$ and $\Theta$ has been defined in \eqref{TTtheta}.
\end{Proposition}

\begin{proof} Take $p \in \partial M$, $\eps$ small enough, $\bar v \in C^{2,\alpha}_{m,NC}$ with small norm, and $b \in \partial \mathbb{R}^{n+1}_+$. We denote by $L_\eps$ the linearization of $F$ with respect to $\bar v$, and by $L^2_\eps$ the second derivative of $F$ with respect to $\bar v$, both computed at the point $(p, \eps, 0)$:
\[
L_\eps = \partial_{\bar v} F(p,\eps,0)\, \qquad \, \textnormal{and}\, \qquad \, L^2_\eps = \partial^2_{\bar v} F(p,\eps,0)\,.
\]
We have
\[
 \int_{S^{n}_+} \, F (p, \eps, \bar v) \, \langle b, \cdot \rangle = \int_{S^{n}_+} \, ( F (p, \eps, 0)  + L_0 \bar v) \, \langle b, \cdot \rangle +  \int_{S^{n}_+} \, (F (p, \eps, \bar v)  - F(p, \eps, 0)  - L_\eps \bar v) \, \langle b, \cdot \rangle \, + \int_{S^{n}_+} \, (L_\eps  - L_0) \bar v \, \langle b, \cdot \rangle
\]
Now we apply this formula for our function $\bar v = \bar v (p, \eps)$ given by Proposition \ref{prop:partsol}. We have $\bar v \in K^\bot$, so $L_0 \, \bar v \in K^\bot$, and then
\[
\int_{S^{n}_+} \,  L_0 \, \bar v  \, \langle b, \cdot \rangle \, = 0\,.
\] 
We obtain that
\begin{equation}\label{interme}
 \int_{S^{n}_+} \, F (p, \eps, \bar v) \, \langle b, \cdot \rangle = \int_{S^{n}_+} \, F (p, \eps, 0) \, \langle b, \cdot \rangle +  \int_{S^{n}_+} \, (F (p, \eps, \bar v)  - F(p, \eps, 0)  - L_\eps \bar v) \, \langle b, \cdot \rangle \, + \int_{S^{n}_+} \, (L_\eps  - L_0) \bar v \, \langle b, \cdot \rangle
\end{equation}
where $\bar v = \bar v(p,\eps)$ is the function given by Proposition \ref{prop:partsol}. We need now two intermediate lemmas. 

\begin{Lemma}
\label{le:3.3}
For all $p \in \partial M$, for all $b = (0,b') \in \partial \mathbb R^{n+1}_+$ we have the following $\eps$-expansion:
\[
\int_{S^{n}_+} F(p,\eps,0) \,  \langle b, \cdot \rangle   =  C \, \eps^{2} \, \tilde g( \nabla^{\tilde g} {\rm H} (p) , \Theta (b') )  + |b|\, \mathcal{O}(\eps^{3}) \,, 
\]
where $\Theta$ is defined in \eqref{TTtheta} and
\[
C = - 2\, \left( \int_{S^{n}_+} y^0\, (y^{1})^{2}  \, \right) \, \frac{1}{\partial_r \phi_1 (1)} \,  \int_{B_{1}^+} r\, |\partial_r \phi_1 |^{2}
\]
where $r = |y|$.
\end{Lemma}

\begin{proof} We recall that 
\[
F (p, \eps, \bar v) =  \displaystyle \hat g (\nabla \hat \phi ,  \hat \nu )  \, |_{\partial B_{1}^+ \cap \mathbb R^{n+1}_+}   -\frac{1}{ {\rm Vol}_{\hat g} (\partial B_{1}^+ \cap \mathbb R^{n+1}_+) } \,  \int_{\partial B_{1}^+ \cap \mathbb R^{n+1}_+}Ê \, \hat g (\nabla \hat \phi ,  \hat \nu ) \, \mbox{dvol}_{\hat g} \, ,
\]
where the metric $\hat g$ has been defined in \eqref{metrichat} for the coordinates $y$. Then 
\[
\int_{S^{n}_+} F(p,\eps,\bar v) \,  \langle b, \cdot \rangle   =   \int_{S^{n}_+} \hat g (\nabla \hat \phi ,  \hat \nu ) \,  \langle b, \cdot \rangle \,.
\]
When $\bar v=0$ we have $\hat \nu = (1+ v_0)\, \partial_r$ on $\partial B_{1}^+ \cap \mathbb R^{n+1}_+$, where $r = |y|$. Then
\begin{equation}\label{interm}
\int_{S^{n}_+} F(p,\eps,\bar v) \,  \langle b, \cdot \rangle   =   (1+ v_0)\, \int_{S^{n}_+} \frac{\partial \hat \phi}{\partial r} \,  \langle b, \cdot \rangle = \frac{1+ v_0}{\partial_r \phi_1 (1) }  \, \displaystyle \int_{S^{n}_+} \frac{\partial \hat \phi}{\partial r} \, \langle \nabla \phi_1 ,  b\rangle 
\end{equation}
where we used the fact that $\phi_1$ is a radial function. Using this last property and the Green's identities we have:
\[
\begin{array}{rllll}
\displaystyle \int_{S^{n}_+}  \frac{\partial \hat \phi}{\partial r} \, \langle \nabla \phi_1 ,  b\rangle  & =  & \displaystyle \int_{B_{1}^+} (\Delta + \lambda_{1}) \hat \phi\,\, \langle \nabla \phi_{1} , b \rangle  - \int_{B_{1}^+} \hat \phi \, \, (\Delta + \lambda_{1} )  \langle \nabla \phi_{1} , b\rangle \\[3mm]
 & =  & \displaystyle \int_{B_{1}^+} (\Delta + \lambda_{1}) \hat \phi\,\, \langle \nabla \phi_{1} , b \rangle    \\[3mm]
& =  & \displaystyle \int_{B_{1}^+} (\Delta - \Delta_{\hat g} ) \,  \hat \phi \,\, \langle \nabla \phi_{1} , b \rangle    + (\lambda_1 - \hat \lambda )  \, \int_{B_{1}^+} \hat \phi \,\,  \langle \nabla \phi_{1} , b \rangle  \\[3mm]
& =  & \displaystyle \int_{B_{1}^+}(\Delta - \Delta_{\hat g} ) \,  \phi_1\,  \langle \nabla \phi_{1} , b \rangle  + \int_{B_{1}^+} (\Delta- \Delta_{\hat g} ) \,  (\hat \phi  - \phi_1)\,\langle \nabla \phi_{1} , b \rangle + (\lambda_1 - \hat \lambda ) \int_{B_{1}^+} (\hat \phi - \phi_1) \, \langle \nabla \phi_{1} , a \rangle \\[1mm]
\end{array}
\]

\medskip

Let compute the first term. Recall that 
\[
\Delta_{\hat g}  : =  \sum_{i,j=0}^n \hat g^{ij} \, \partial_{y_i} \partial_{y_j} + \sum_{i,j=0}^n \partial_{y_i}  \hat g^{ij} \, \partial_{y_j} + \frac{1}{2} \, \sum_{i,j=0}^n \hat g^{ij} \, \partial_{y_i} \log |\hat g| \, \partial_{y_j}   \,.
\]
From \eqref{p000} we have that the coefficients of the metric $\hat g$ can be expanded, for $i,k,j,\ell = 1,...,n$, as 
\begin{eqnarray*}
\hat g_{00}(y) & = & (1+ v_0)^2\\[1mm]
\hat g_{0j}(y)& = & 0\\[1mm]
\hat g_{ij}(y) & = & (1+ v_0)^2\, \Bigg( \delta_{ij} + 2(1+ v_0)\,\eps\, g(\nabla_{E_{i}}N,E_{j})\,y^0 + R_{0i0j}\,(1+ v_0)^2\,\eps^2 \, (y^{0})^2\\
& & +(1+ v_0)^2\,\eps^2\, g(\nabla_{E_{i}}N,\nabla_{E_{j}}N)\,(y^0)^2 + 2 (1+ v_0)^2\,\eps^2\,\sum_{k} R_{k0ij}\, y^{k} \, y^{0}\\
& & + \frac{1}{3}\,(1+ v_0)^2\,\eps^2 \, \sum_{k,\ell} \tilde R_{ikj\ell} \, y^{k} \, y^{\ell} + \mathcal{O}(\eps^3)\Bigg)
\end{eqnarray*}
Keeping in mind that $v_0 = v_0(p,\eps) = \mathcal O (\eps)$, the third equality simplifies slightly obtaining
\begin{eqnarray*}
\hat g^{00}(y) & = & (1+ v_0)^{-2}\\[1mm]
\hat g^{0j}(y)& = & 0\\
\hat g^{ij}(y) & = & (1+ v_0)^{-2} \Bigg( \delta_{ij} - 2(1+ v_0)\,\eps\, g(\nabla_{E_{i}}N,E_{j})\,y^0 - R_{0i0j}\,\eps^2 \, (y^{0})^2\\
& & - \eps^2\, g(\nabla_{E_{i}}N,\nabla_{E_{j}}N)\,(y^0)^2 - 2 \eps^2\,\sum_{k} R_{k0ij}\, y^{k} \, y^{0}-\frac{1}{3}\,\eps^2 \, \sum_{k,\ell} \tilde R_{ikj\ell} \, y^{k} \, y^{\ell}\Bigg)\, + \mathcal{O}(\eps^3)\,.
\end{eqnarray*}
Using the fact that $R_{k0ii} = 0$, we have
\[
\begin{array}{rllll}
\log |\hat g| & = & \displaystyle 2n \, \log (1+v_0)  - 2\,\eps(1+ v_0)\, {\rm H}(p) \,y^0\\[1mm] 
& & \displaystyle +\, \frac{\eps^2}{2}\,\Bigg\{ \Bigg[-{\rm Ric}(N) + 4 \sum_{i\neq j} g(\nabla_{E_{i}}N,E_{i})\, g(\nabla_{E_{j}}N,E_{j}) + \sum_{i}g(\nabla_{E_{i}}N,\nabla_{E_{i}}N)\, \\[1mm] 
& & \displaystyle \qquad -\, 4 \sum_{i\neq j} g(\nabla_{E_{i}}N,E_{j}) \, g(\nabla_{E_{j}}N,E_{i})\Bigg](y^{0})^2\, +\, \frac{1}{3} \, \tilde R_{k\ell} \, y^{k} \, y^{\ell}\Bigg\}\, + {\mathcal O} (\eps^{3}) 
\end{array}
\]
where $\rm Ric$ denotes the Ricci curvature of $\partial M$ and
$$
\tilde R_{k\ell} = \sum_{i=1}^{n} \tilde R_{iki\ell}\,.
$$
A straightforward computation {(still keeping in mind that $v_0 = \mathcal O (\eps)$)} shows that 
\[
\begin{array}{rllll}
 \displaystyle \big(\Delta- \Delta_{\hat {g}} \big)\phi_{1}  & = & \displaystyle  -\, \lambda_1 \, (1- (1+ v_0)^{-2}) \, \phi_1 \\[1mm] 
 & & \displaystyle +\, 2\,(1+ v_0)^{-1}\, \eps\, \sum_{i,j} g(\nabla_{E_{i}}N,E_{j})\, y^0\, \left( \frac{y^{i}y^{j}}{r^{2}} \, \partial_r^2 \phi_1 + \frac{\delta_j^i}{r} \, \partial_r \phi_1  -\, \frac{y^{i} y^{j}}{r^{3}} \, \partial_r \phi_1 \right)\\[1mm] 
 & & \displaystyle +\,\eps\,(1+ v_0)^{-1}\, {\rm H}(p) \,\frac{y^0}{r}\partial_r \phi_1 \\[3mm]
 & & \displaystyle +\, \, \eps^2\, \sum_{k,i,j, \ell} \left[ \left[R_{0i0j} + g(\nabla_{E_{i}}N,\nabla_{E_{j}}N)\right] (y^0)^2\, +\, 2\, R_{k0ij}\, y^{k} y^{0} +\, \frac{1}{3} \, \tilde R_{ikj\ell}\, y^{k} y^{\ell} \right] \cdot \\[1mm]
 & & \hspace{1cm}\displaystyle \qquad \cdot\,\left( \frac{y^{i}y^{j}}{r^{2}} \, \partial_r^2 \phi_1 + \frac{\delta_j^i}{r} \, \partial_r \phi_1  - \frac{ y^{i} y^{j}}{r^{3}} \, \partial_r \phi_1 \right)\\[5mm] 
 & & \displaystyle +\, \eps^2\, \sum_{k,i,j}\, \left( 2  R_{i0ij} \, y^{0}\, +\, \frac{1}{3} \, \tilde R_{ikji} \, y^{k}\, +\, \frac{1}{6} \, \tilde R_{ik} \, y^{k}\right)\, \frac{y^j}{r}\partial_r \phi_1\\[5mm]
& & \displaystyle +\, \eps^2\,\Bigg[-{\rm Ric}(N) + 4 \sum_{i\neq j} g(\nabla_{E_{i}}N,E_{i})\, g(\nabla_{E_{j}}N,E_{j})\, + \sum_{i} g(\nabla_{E_{i}}N,\nabla_{E_{i}}N) \\[1mm]
& & \hspace{1cm}\displaystyle \qquad  - 4 \sum_{i\neq j} g(\nabla_{E_{i}}N,E_{j})\, g(\nabla_{E_{j}}N,E_{i})\Bigg] \cdot\,\frac{(y^0)^2}{r}\partial_r \phi_1
\end{array}
\]
where $i,j,k = 1,...,n$.
Observe that we have used the fact that $R(X,X)\equiv 0$ and the symmetries of the curvature tensor for which $R_{ijkl} = R_{klij}$. Now, in the computation of 
\[
\int_{B_{1}^+}(\Delta - \Delta_{\hat g} ) \,  \phi_1\,  \langle \nabla \phi_{1} , b \rangle \,,
\]
observe that the terms in the expansion of $(\Delta- \Delta_{\hat {g}} ) \, \phi_{1}$ which contain an even number of coordinates different to $y^0$, such as $y^0$ or $y^{i}y^{j}y^{k}y^{\ell}$ or $(y^0)^2 y^{i}y^{j}$ etc.  do not contribute to the result  since, once multiplied by $\langle \nabla \phi_1  , b\rangle$ (keep in mind that $b = (0,b')$), their average over $S^{n}_+$ is $0$.  Therefore, we can write
\[
\begin{array}{rlllll}
 \displaystyle \int_{B_{1}^+} \left(\Delta - \Delta_{\hat g} \right) \phi_{1}\,  \langle \nabla \phi_1, \, b\rangle & = & \displaystyle \eps^{2} \, \sum_{\sigma \neq 0}\, \int_{{B_{1}^+}} \, \partial_r \phi_1 \,  a_{\sigma}  \frac{y^{\sigma}}{r} \cdot \\[1mm]
 &  &  \displaystyle \qquad \cdot \Bigg(\displaystyle  2\,\, \sum_{k,i,j} R_{k0ij} \, \left( \frac{y^{i}y^{j} y^{k} y^{0}}{r^{2}} \, \partial_r^2 \phi_1 - \frac{y^{i}y^{j} y^{k} y^{0}}{r^{3}} \, \partial_r \phi_1 \right)  +2\, \sum_{k,i,j}\,  R_{i0ij} \, \frac{y^{0}y^j}{r}\partial_r \phi_1\Bigg) \\[5mm]
 & & +\, \mathcal{O}(\eps^{3})
\end{array}
\]
We make use of the technical Lemmas \ref{tec1} and \ref{tec2} of the Appendix to conclude that
\begin{equation}\label{abc}
 \displaystyle \int_{B_{1}^+} \left(\Delta - \Delta_{\hat g} \right) \phi_{1} \, \langle \nabla \phi_1 , b\rangle  = \tilde C \, \eps^{2} \, {\tilde g \big(\nabla^{\tilde g} {\rm H(p)},  \Theta (b') \big)}\, +\, \mathcal{O}(\eps^{3}).
\end{equation}
where
\[
\tilde C = -2\, \left( \int_{S^{n}_+} y^0\, (y^{1})^{2} \right) \, \int_{B_{1}^+} r\, |\partial_r \phi_1 |^{2} \, .
\]

\medskip

Now we have to compute the terms
\[
\int_{B_{1}^+}(\Delta - \Delta_{\hat g} ) \,  (\hat \phi  - \phi_1)\,  \langle \nabla \phi_{1} , b \rangle \qquad \textnormal{and} \qquad
(\lambda_1 - \hat \lambda )  \, \int_{B_{1}^+} (\hat \phi - \phi_1) \,  \langle \nabla \phi_{1} , a \rangle \,.
\]
We observe that the coefficients of the metric, for $i,j = 1,...,n$, are given by
\[
\hat g_{ij}(y) = \delta_{ij} + \hat G_{ij}\, \eps\, y^0 + \mathcal{O}(\eps^2)
\]
for some constants $G_{ij}$. Then the $\eps$-first order term of $\hat \phi  - \phi_1$ is radial in the coordinates $y^1,..., y^n$, i.e. there exists a function $h \in C^{2,\alpha}([0,1]^2)$ such that 
\[
(\hat \phi  - \phi_1)(y^0,y') = \eps\, h(y^0,|y'|) + \mathcal{O}(\eps^2)\,.
\]
Let $\rho:=|y'|$. Using the same computation given above, we find
\[
\begin{array}{rllll}
 \displaystyle \big(\Delta - \Delta_{\hat {g}} \big)(\hat \phi  - \phi_1)  & = & \displaystyle  (1- (1+ v_0)^{-2}) \, \Delta (\hat \phi  - \phi_1)  \\[2mm] 
 & & \displaystyle +\, \mathcal{O}(\eps^2)\, \left( \frac{y^0y^{i}y^{j}}{\rho^{2}} \, \partial_\rho^2 h + \frac{y^{0}}{\rho} \delta_j^i \, \partial_\rho h  -\, \frac{y^{0}y^{i} y^{j}}{\rho^{3}} \, \partial_\rho h +\, \partial_{y^0} h\right) + \mathcal{O}(\eps^3)\\[2mm]
 & = & \displaystyle \mathcal{O}(\eps^2)\, \left(\tilde h(y^0,\rho) \, + \frac{y^0y^{i}y^{j}}{\rho^{2}} \, \partial_\rho^2 h + \frac{y^{0}}{\rho} \delta_j^i \, \partial_\rho h  -\, \frac{y^{0}y^{i} y^{j}}{\rho^{3}} \, \partial_\rho h +\, \partial_{y^0} h\right) + \mathcal{O}(\eps^3)
\end{array}
\]
for some function $\tilde h \in C^{0,\alpha}([0,1]^2)$, and the terms $\mathcal{O}(\eps^2)$ do not depend on the coordinates. As in the previous computation, terms which contain an even number of coordinates different to $y^0$ do not contribute to the result since, once multiplied by $ \langle \nabla \phi_1  , b\rangle$, their average over $S^{n}_+$ is $0$. Therefore
\[
\displaystyle \int_{B_{1}^+} (\Delta- \Delta_{\hat g} )  (\hat \phi  - \phi_1)\,\, \langle \nabla \phi_{1} , b\rangle  = \mathcal{O}(\eps^3).
\]
For the last term we have to estimate, the previous computation immediately implies that
\[
\int_{B_{1}^+}  (\hat \phi - \phi_1)\,\, \langle \nabla \phi_{1} , b \rangle = \mathcal{O}(\eps^2)
\]
and then 
\[
 (\lambda_1 - \hat \lambda )  \, \int_{B_{1}^+}  (\hat \phi - \phi_1)\,\, \langle \nabla \phi_{1} , b \rangle = \mathcal{O}(\eps^3)\,.
 \]
 
 \medskip
 
We conclude that 
\[
\int_{S^{n}_+} \frac{\partial \hat \phi}{\partial r} \, \langle \nabla \phi_1 , b\rangle  = \int_{B_{1}^+} (\Delta - \Delta_{\hat g} ) \phi_1\,\, \langle \nabla \phi_{1} , b \rangle + |b| \, \mathcal{O}(\eps^3) = \tilde C \, \eps^{2} \, {\tilde g \big(\nabla^{\tilde g} {\rm H(p)},  \Theta (b') \big)}\, + |b|\, \mathcal{O}(\eps^{3}) \,.
\]
The Lemma follows at once from \eqref{interm}, keeping in mind that $v_0 = \mathcal{O}(\eps)$.
\end{proof}

\begin{Lemma}\label{le:4444} Let $\bar v = \bar v(p, \eps) \in C^{2,\alpha}_{m,NC} (S^n_+)$ such that in the coordinates $y=(y^0,y')$ we have
\[
\bar v (y^0,y') =  \eps \, \tilde v_p (y^0) + \mathcal{O}(\eps^2)
\]
for some function $\tilde v_p \in C^{2,\alpha}([0,1])$. Then there exist two functions $\delta_p, \sigma_p  \in C^{2,\alpha}([0,1])$ such that
\[
((L_\eps  - L_0) \, \bar v ) (y^0,y') = \eps^2\, \delta_p(y^0) + \mathcal O (\eps^3)\,
\]
and
\[
F (p, \eps, \bar v)  - F(p, \eps, 0)  - L_\eps \bar v  = \eps^2\, \sigma_p(y^0) + \mathcal O (\eps^3)\,.
\]
\end{Lemma}

\begin{proof} Clearly both $L_\eps$ and $L_0$ are first order differential operators, and the dependence on $\eps$ is smooth. Now, the difference between the coefficients of $\bar g$ written in the coordinates $y$ defined in \eqref{p000} and the coefficient of the Euclidean metric can be estimated by 
\[
\bar g_{ij}(y^0,y') = \bar G_{ij}\, \eps\, y^0 + \mathcal O (\eps^2)
\]
If the function $\bar v$ is such that 
\[
\bar v (y^0,y') =  \eps \, \tilde v_p (y^0) + \mathcal{O}(\eps^2)
\]
for some function $\tilde v_p \in C^{2,\alpha}([0,1])$, it is then clear that 
\[
((L_\eps  - L_0) \, \bar v ) =  \eps\, ((L_\eps  - L_0) \, \tilde v_p ) + \mathcal{O}(\eps^3)
\]
where now the function $\tilde v_p$ is considered as a function on the coordinates $(y^0,y')$ by the simple relation $\tilde v_p (y^0,y')=\tilde v_p(y^0)$. Moreover if we consider the operator $F$ restricted to functions $\bar v$ that depend only on the first variable $y^0$, it is clear that the linearization of $F$ at $(p,\eps,0)$ maps from the subset of functions in $C^{2,\alpha}_{m,NC}$ that depend only on the first variable $y^0$ into the subset of functions in $C^{1,\alpha}_{m,NC}$ that depend only on the first variable $y^0$. Then there exists a function $\delta_p \in C^{1,\alpha}([0,1])$ such that
\[
((L_\eps  - L_0) \, \tilde v_p) (y^0,y') = \eps\, \delta_p(y^0) + \mathcal O (\eps^2)\,
\]
and then
\[
((L_\eps  - L_0) \, \bar v ) (y^0,y') = \eps^2\, \delta_p(y^0) + \mathcal O (\eps^3)\,.
\]

\medskip

Now let us estimate the second term. Taking in account that $\bar v = \mathcal{O}(\eps)$ we have
\[
F (p, \eps, \bar v)  = F(p, \eps, 0)  +  L_\eps \bar v  + L_\eps^2 (\bar v, \bar v) + \mathcal{O}(\eps^3)
\]
and then 
\[
F (p, \eps, \bar v)  - F(p, \eps, 0)  -  L_\eps \bar v  = L_\eps^2 (\bar v, \bar v) + \mathcal{O}(\eps^3)\,.
\]
If the function $\bar v$ is such that 
\[
\bar v (y^0,y') =  \eps \, \tilde v_p (y^0) + \mathcal{O}(\eps^2)
\]
then
\[
F (p, \eps, \bar v)  - F(p, \eps, 0)  -  L_\eps \bar v  = \eps^2\, L_\eps^2 (\tilde v_p, \tilde v_p) + \mathcal{O}(\eps^3)\,
\]
where again the function $\tilde v_p$ is considered as a function on the coordinates $(y^0,y')$ by $\tilde v (y^0,y')=\tilde v(y^0)$, and as for $L_\eps$ it is easy to see that $L_\eps^2$ maps from the subset of functions in $C^{2,\alpha}_{m,NC}$ that depend only on the first variable $y^0$ into the subset of functions in $C^{1,\alpha}_{m,NC}$ that depend only on the first variable $y^0$. Then there exists a function $\sigma_p \in C^{1,\alpha}([0,1])$ such that
\[
F (p, \eps, \bar v)  - F(p, \eps, 0)  - L_\eps \bar v  = \eps^2\, \sigma_p(y^0) + \mathcal O (\eps^3)\,.
\]
This completes the proof of the Lemma.
\end{proof}

We are now able to conclude the proof of Proposition \ref{kernelproj}. Using Lemma \ref{le:4444} we get 
\[
\int_{S^{n}_+} \, (F (p, \eps, \bar v)  - F(p, \eps, 0)  - L_\eps \bar v) \, \langle b, \cdot \rangle + \int_{S^{n}_+} \, (L_\eps  - L_0) \bar v \, \langle b, \cdot \rangle = \mathcal O (\eps^3)\,.
\]
Then, from \eqref{interme} and using Lemma \ref{le:3.3}, we have that for all $p \in \partial M$ and all $b \in \partial \mathbb{R}^{n+1}_+$ with $|b|=1$ the following $\eps$-expansion holds:
\[
\int_{S^{n}_+} F(p,\eps,\bar v(p,\eps)) \,  \langle b, \cdot \rangle  =  C \,\eps^2 \, \tilde g( \nabla^{\tilde g} {\rm H} (p) , \Theta (b') )  + \mathcal{O}(\eps^3) \,.
\]
This completes the proof of the Proposition.
\end{proof}

\subsection{Proof of Theorem \ref{maintheorem}}\label{ssect:mainproof}
Let $b = (0,b') \in \partial \mathbb{R}^{n+1}_+$ with $|b|=1$ and define
\[
G_b(p,\eps) : = \eps^{-2}\, \int_{S^{n}_+} F(p,\eps,\bar v(p,\eps)) \,  \langle b, \cdot \rangle  =  C \, \tilde g( \nabla^{\tilde g} {\rm H} (p) , \Theta (b') )  + \mathcal{O}(\eps) \,.
\]
Clearly if $\eps \neq 0$, we have that
\[
\int_{S^{n}_+} F(p,\eps,\bar v(p,\eps)) \,  \langle b, \cdot \rangle = 0   \qquad \Longleftrightarrow \qquad G_b(p, \eps) =0 \,.
\]
$G_b$ is a function defined on $\partial M \times [0,+\infty)$ into $\mathbb{R}$. By the assumption of our main Theorem \ref{maintheorem}, $\partial M$ has a nondegenerate critical point $p_0$ of the mean curvature. Then the differential of $G_b$ with respect to $p$ computed at $(p_0, 0)$ is invertible and $G_b(p_0,0) =0$. By the implicit function theorem, for all $\eps$ small enough there exists $p_\eps \in \partial M$ close to $p_0$ such that 
\[
G_b(p_\eps, \eps) =0
\]
for all $b \in \partial \mathbb{R}^{n+1}_+$ with $|b|=1$. In addition we have 
\[
\mbox{dist} (p_0 , p_\eps ) \leq c \, \eps
\]
We conclude then that 
\[
F(p_\eps,\eps,\bar v(p_\eps,\eps)) \in K^\bot
\]
where $K$ is the kernel of the operator $L_{0}$. But by the construction of $\bar v$, we have also that
\[
F(p_\eps,\eps,\bar v(p_\eps,\eps)) \in K
\]
and then 
\[
F(p_\eps,\eps,\bar v(p_\eps,\eps)) =0 \,.
\]
This means that the normal derivative of the first eigenfunction of the Laplace-Beltrami operator on $\Omega_\eps = B^+_{g,\eps} (p_\eps)$ with mixed boundary condition is constant on $\partial \Omega_\eps \cap \mathring M$ and then $\Omega_\eps$ is extremal. 
\medskip

The only remaining point in the proof of Theorem~\ref{maintheorem}, is the analyticity of $\partial\Om_{\eps}\cap\mathring{M}$ when $M$ itself is analytic. This is a classical consequence of the extremality condition, see \cite{KN}.

\section{Appendix}

\subsection{Expansion of the metric}

Take the local coordinates $x^0, x^1, ..., x^n$ in a neighborhood of a point $p \in \partial M$ that we introduced in \eqref{eq:normalcoord}. We denote the corresponding coordinate vector fields by
\[
X_{j}:= \Psi_{*}(\partial_{x^j})
\]
for $j = 0, 1, ..., n$. We want to write the expansion of the coefficients $g_{ij}$ of the metric $\Psi^* g$ in these coordinates. According with our notation, $E_{j}$ are the coordinate vector field $X_{j}$ evaluated at $p$.

\begin{Proposition}\label{fermi-exp}
At the point of coordinate $x = (x^0, x^1, ..., x^n)$, the following expansion holds~:
\begin{eqnarray*}
g_{00} & = & 1\\[1mm]
g_{0j} & = & 0\\[1mm]
g_{ij} & = & \delta_{ij}\, + 2\,g(\nabla_{E_{i}}N,E_{j})\,x^0\, + R_{0i0j} \, (x^{0})^2 + g(\nabla_{E_{i}}N,\nabla_{E_{j}}N)\,(x^0)^2 \\
& & + 2\, \sum_{k} R_{k0ij} \, x^{k} \, x^{0} + \frac{1}{3} \, \sum_{k,\ell} \tilde{R}_{ijkl} \, x^{k} \, x^{\ell} + \mathcal{O}(|x|^3)
\end{eqnarray*}
for $i,j,k,l = 1,...n$, where
\begin{eqnarray*}
R_{0i0j} & = & g\big( R(N, E_{i}) \, N ,E_{j}\big)\\
R_{k0ij} & = & g\big( R(E_{k}, N) \, E_i ,E_{j}\big)\\
\tilde R_{ijkl} & = & \tilde g\big( \tilde R(E_{i}, E_{k}) \, E_{j} ,E_{\ell}\big).
\end{eqnarray*}
Here $R$ and $\tilde R$ are respectively the curvature tensors of $M$ and $\partial M$.
\label{pr:1.1}
\end{Proposition}

This result of this proposition is very well known. For example, the same kind of coordinates that we use in this paper are also used in \cite{Pac-Sun}, and Proposition 5.1 of \cite{Pac-Sun} combined with the classical expansion of a metric in its geodesic normal coordinate (see for example \cite{willmore}) immediately implies our Proposition \ref{fermi-exp}. Nevertheless, in order to make the reading easier, we write the proof of the proposition.

\medskip

\begin{proof} We consider the mapping $F$.
The curve $x^0 \longmapsto F(x^0, x)$ being a geodesic we have
$g(X_0, X_0) \equiv 1$. This also implies that  $\nabla_{X_0} X_0 \equiv
0$  and hence  we get
\[
\partial_{x^0} g(X_0, X_j) = g ( \nabla_{X_0} \, X_0, X_j ) + g (
\nabla_{X_0} \, X_j, X_0 ) = g ( \nabla_{X_0} \, X_j, X_0 ) \, .
\]
The vector fields $X_0$ and $X_j$ being coordinate vector fields we have
$\nabla_{X_0} X_j = \nabla_{X_j} X_0$ and we conclude that
\[
2 \, \partial_{x^0} g(X_0, X_j) =  2 \, g ( \nabla_{X_j} \, X_0, X_0 ) =
\partial_{x^j} g(X_0, X_0) = 0 \, .
\]
Therefore, $g(X_0, X_j)$ does not depend on $x^0$ and since on $\partial M$
this quantity is $0$ for $j=1, \ldots, n$, we conclude that the metric
$g$ can be written as
\[
g = d(x^0)^2 + \bar g_{x^0} \, ,
\]
where $\bar g_{x^0}$ is a family of metrics on $\partial M$ smoothly
depending on $x^0$ (this is nothing but Gauss' Lemma). If
$\tilde g$ is the metric of $\partial M$ induced by $g$, we certainly  have
\[
\bar g_{x^0} =\tilde g + {\mathcal O} (x^0) \, .
\]

\medskip

We now derive the next term the expansion of $\bar g_{x^0}$ in powers of
$x^0$. To this aim, we compute
\[
\partial_{x^0} \, g  (X_i, X_j) =  g ( \nabla_{X_i} \, X_0, X_j ) +  g (
\nabla_{X_j} \, X_0, X_i ) \, ,
\]
for all $i,j = 1, \ldots, n$.   Since $X_0 =N$ on $\partial M$, we get
\[
\partial_{x^0} \, \bar g_{x^0} \, |_{x^0=0} = 2 \, g (\nabla_{\cdot} \, N, \cdot ) \, ,
\]
by definition of the second fundamental form.  This already implies that
\[
\bar g_{x^0} = \tilde g + 2 g (\nabla_{\cdot} \, N, \cdot )\, \, x^0 + {\mathcal O}
((x^0)^2) \, .
\]

\medskip

Using the fact that the $X_0$ and $X_j$ are coordinate vector fields, we
can compute
\begin{equation}
\partial_{x^0}^2 \, g (X_i, X_j) =  g( \nabla_{X_0} \, \nabla_{X_i}\,
X_0, X_j ) + g( \nabla_{X_0} \, \nabla_{X_j}\, X_0, X_i ) +  2 \, g(
\nabla_{X_i} \, X_0 , \nabla_{X_j} X_0 ).
\label{eq:1-11}
\end{equation}
By definition of the curvature tensor, we can write
\[
\nabla_{X_0} \, \nabla_{X_j} = R(X_0 , X_j) + \nabla_{X_j} \,
\nabla_{X_0} + \nabla_{[X_0, X_j]} \, ,
\]
which, using the fact that $X_0$ and $X_j$ are coordinate vector fields,
simplifies into
\[
\nabla_{X_0} \, \nabla_{X_j} = R(X_0 , X_j) + \nabla_{X_j} \,
\nabla_{X_0} \, .
\]
Since $\nabla_{X_0} \, X_0 \equiv 0$, we get
\[
\nabla_{X_0} \, \nabla_{X_j} X_0 = R(X_0 , X_j) \, X_0 \, .
\]
Inserting this into (\ref{eq:1-11}) yields
\[
\partial_{x^0}^2 \, g (X_i, X_j)  =  2 \, g( R(X_0, X_i) \, X_0, X_j ) +
2 \, g( \nabla_{X_i} \, X_0,  \nabla_{X_j} X_0 ) \, .
\]
Evaluation at $x^0=0$ gives
\[
\partial_{x^0}^2 \,  \bar g_{x^0} \, |_{x^0=0} = 2 \, g ( R(N,  \cdot )
\, N, \cdot )  + 2 \,g( \nabla_{\cdot} \, N ,  \nabla_{\cdot} N ) .
\]
This implies that 
\begin{equation}\label{expansion}
\bar g_{x^0} =  \tilde g + 2 g (\nabla_{\cdot} \, N, \cdot )\, \, x^0 + \left[g (\nabla_{\cdot} \, N,
\nabla_{\cdot} \, N) + g( R( N, \cdot  )\, N, \cdot )\right]\, (x^0)^2 + {\mathcal O}
((x^0)^3)
\end{equation}

\medskip 

Now that we have the first terms of the expansion of $\bar g_{x^0}$ in powers of
$x^0$ we find the expansion of these term with respect to the geodesic coordinates $(x^1,...,x^n)$ of $\partial M$ in a neighborhood of $p$. Recall that for $i,j,k,l = 1,...,n$
\begin{equation}\label{taylor}
\tilde g_{ij} = \delta_{ij} + \frac{1}{3} \, \sum_{k,\ell} \tilde R_{ikj\ell} \, x^{k} \, x^{\ell} + {\mathcal O}(|x|^{3}),
\end{equation}
where
\[
\tilde R_{ikj\ell} = \tilde g \big( \tilde R(E_{i}, E_{k}) \, E_{j} ,E_{\ell}\big)
\]
The proof of this fact can be found for example in \cite{willmore}. Moreover for $k = 1,...,n$ we have
\begin{eqnarray*}
\partial_{x^k} g (\nabla_{X_{i}} \, N, X_j ) & = & g (\nabla_{X_k}\nabla_{X_{i}} \, N, X_j ) + g (\nabla_{X_{i}} \, N, \nabla_{X_k} X_j )\\[1mm]
& = & g (\nabla_{X_k}\nabla_{N} \, X_i, X_j ) + g (\nabla_{X_{i}} \, N, \nabla_{X_k} X_j )\\[1mm]
& = & g(R(X_k,N)X_i,X_j) + g (\nabla_{N}\nabla_{X_k} \, X_i, X_j ) + g (\nabla_{X_{i}} \, N, \nabla_{X_k} X_j )
\end{eqnarray*}
and evaluated at $p$
\begin{equation}\label{primo}
\partial_{x^k} g (\nabla_{X_{i}} \, N, X_j ) |_p = g(R(E_k,N)E_i,E_j)
\end{equation}
From (\ref{expansion}), using (\ref{taylor}) and (\ref{primo}), we find the expansion of the metric in the coordinates $x^0,x^1,...,x^n$ up to the term of order $|x|^2$.
 \end{proof}

\subsection{Technical Lemmas}

\begin{Lemma}\label{tec1}
For all $\sigma =1, \ldots, n$, we have
$$
\sum_{i,j, k} \int_{S^{n}_+} R_{k0ij} \, x^0\, x^{i} \, x^{j} \, x^{k} \, x^{\sigma}  = 0.
$$
\end{Lemma}
\begin{proof}
 To see that we consider all terms of the above sum, obtained fixing the $4$-tuple $(i,k,j , \sigma)$. We observe that if in such a $4$-tuple there is an element that appears an odd number of time then $\displaystyle \int_{S^{n}_+} \, x^0\, x^{i} \, x^{j} \, x^{k} \, x^{\sigma} = 0$. Then
\[
\sum_{i,j, k} \int_{S^{n}_+} R_{k0ij} \, x^0\, x^{i} \, x^{j} \, x^{k} \, x^{\sigma} = \sum_{i} \int_{S^{n}_+} \Big( R_{\sigma0ii} + R_{i0i\sigma} + R_{i0\sigma i}\Big) \, x^0\, (x^{i})^2 \, (x^{\sigma})^2 = 0
\]
by the symmetries of the curvature tensor.
\end{proof}

\medskip

\begin{Lemma}\label{tec2}
For all $\sigma =1, \ldots, n$, we have
$$
\displaystyle \sum_{i,j}  \int_{S^{n}_+} R_{i0ij} \, x^0\, x^{j} \, x^{\sigma} = - \left( \int_{S^{n}_+} x^0\, (x^{1})^{2}\right) \, {\rm H}_{,\sigma}
$$
\end{Lemma}
\begin{proof}
Again, we find that $\displaystyle \int_{S^{n}_+} \, x^0\, x^{j}\, x^{\sigma} \, {\rm dvol}_{\mathring g} = 0$ unless the indices $j,\sigma$ are equal. Hence
\begin{eqnarray*}
\displaystyle \sum_{i,j}  \int_{S^{n}_+} R_{i0ij} \, x^0\, x^{j} \, x^{\sigma}&  = & \displaystyle \left( \int_{S^{n}_+} x^0\, (x^{\sigma})^{2} \right) \, \sum_{i} R_{i0i\sigma} = - \displaystyle \left( \int_{S^{n}_+} x^0\, (x^{1})^{2}  \right) \, {\rm H}_{,\sigma}
\end{eqnarray*}
This completes the proof of the result.
\end{proof}

\noindent\textbf{Acknowledgements}. This work was partially supported by the project Projet ANR-12-BS01-0007 OPTIFORM financed by the French Agence Nationale de la Recherche (ANR).

\end{document}